\documentclass[8pt,a4paper]{article}

\usepackage{amsmath,amsfonts,amssymb}
\usepackage{bbm}
\usepackage{cases}
\usepackage{kotex}
\usepackage{lipsum}
\usepackage{subcaption}

\usepackage{xcolor}
\usepackage{graphicx}
\usepackage{ulem}
\usepackage{subcaption}

\newtheorem{defn}{Definition}[section]
\newtheorem{theo}[defn]{Theorem}
\newtheorem{lem}[defn]{Lemma}
\newtheorem{prop}[defn]{Proposition}
\newtheorem{cor}[defn]{Corollary}
\newtheorem{rem}[defn]{Remark}

\newenvironment{proof}{{\bf Proof }}{{\vskip 0.1cm \hfill$\Box$}}

\begin{document} 

\noindent
{\Large \bf Analysis of linear elliptic equations with general drifts and $L^1$-zero-order terms}
\\ \\
\bigskip
\noindent
{\bf Haesung Lee}  \\
\noindent
{\bf Abstract.} This paper provides a detailed analysis of the Dirichlet boundary value problem for linear elliptic equations in divergence form with $L^p$-general drifts, where $p \in (d, \infty)$, and non-negative $L^1$-zero-order terms. Specifically, by transforming the general drifts into weak divergence-free drifts, we establish the existence and uniqueness of a bounded weak solution, showing that the zero-order term does not influence the quantity of the unique weak solution. Additionally, by imposing the $ VMO$ condition and mild differentiability on the diffusion coefficients and assuming an $L^s$-zero-order terms with $s \in (1, \infty)$, we demonstrate the existence and uniqueness of a strong solution for the corresponding non-divergence type equations. An important feature of this paper is that, due to the weak divergence-free property of the drifts in the transformed equations, the constants appearing in our estimates can be explicitly calculated,  which is expected to offer significant applications in error analysis. \\ \\
\noindent
{Mathematics Subject Classification (2020): {Primary: 35J25, 35R05, Secondary: 35B65, 35B35}}\\

\noindent 
{Keywords: Linear elliptic equations, boundary value problems, existence and uniqueness, contraction estimates, weak divergence-free drifts, error analysis
}

\section{Introduction} 
This paper mainly deals with the existence and uniqueness of a weak solution (see \eqref{weakdefzerve} and Definition \ref{defnofweaksol}) to the following Dirichlet boundary value problem for linear elliptic equations in a bounded open subset $U$ of $\mathbb{R}^d$ in divergence form:
\begin{equation} \label{maineq2}
\left\{
\begin{alignedat}{2}
-{\rm div}(A \nabla u) + \langle \bold{H}, \nabla u \rangle + c u &=f  && \quad \; \mbox{\rm in $U$,}\\
u &= 0 &&\quad \; \mbox{\rm on $\partial U$},
\end{alignedat} \right.
\end{equation}
where it is assumed that $A=(a_{ij})_{1 \leq i,j \leq d}$ is uniformly strictly elliptic and bounded (i.e. \eqref{elliptici} holds), $\bold{H} \in L^p(U, \mathbb{R}^d)$ with $p \in (d, \infty)$, $c \in L^1(U)$ with $c \geq 0$ in $U$ and $f \in L^q(U)$ with $q \in (\frac{d}{2}, \infty)$.
 As a counterpart to the divergence form, we additionally study the existence and uniqueness of a strong solution (see \eqref{strongformiden} and Definition \ref{defnstrongsol}) to the following Dirichlet boundary value problem for linear elliptic equations in non-divergence form:
\begin{equation} \label{maineqnondiv}
\left\{
\begin{alignedat}{2}
-\text{\rm trace}(A \nabla^2 u) + \langle \widehat{\bold{H}}, \nabla u \rangle + c u &=f  && \quad \;\mbox{\rm in $U$,}\\
u &= 0 &&\quad \;\mbox{\rm on $\partial U$},
\end{alignedat} \right.
\end{equation}
where it is additionally assumed that $\partial U$ is of class $C^{1,1}$, $a_{ij} \in VMO_{\omega}$ for all $1 \leq i,j \leq d$, $\text{div} A \in L^p(U,\mathbb{R}^d)$ (see Definitions \ref{divecdefn}, \ref{vmodefn}), $\widehat{\bold{H}} \in L^p(U, \mathbb{R}^d)$ and $c \in L^{s}(U)$ with $s \in (1, \infty)$. The specific conditions will be provided in our main results in this introduction. \\
Let us briefly review previous results for the existence and uniqueness of a weak solution to \eqref{maineq2}, where the essential approach is based on a continuous bilinear form on a Hilbert space and the Lax-Milgram theorem (\cite[Corollary 5.8]{Br11}). For instance, under the assumption that $A$ is uniformly strictly elliptic and bounded (i.e. \eqref{elliptici} holds), $\bold{H} \in L^d(U, \mathbb{R}^d)$, $c \in L^{\frac{d}{2}}(U)$ with $c \geq 0$ in $U$, $f \in L^{\frac{2d}{d+2}}(U)$
if $d \geq 3$ and $\bold{H} \in L^{r}(U, \mathbb{R}^d)$, $c \in L^{\frac{r}{2}}(U)$ with $r>2$ and $c \geq 0$ in $U$, $f \in L^{\frac{r}{2}}(U)$ if $d=2$, one can check in \cite{S65} (cf. \cite[Lemma 2.1.3]{BKRS15}) the existence of a constant $\gamma_0>0$ which only depends on $d$, $\lambda$, $U$ and $\bold{H}$, such that if $c \geq \gamma_0$ in $U$, then the existence and uniqueness of a weak solution to \eqref{maineq2} are guaranteed.
Moreover, additionally assuming that $c \in L^{r}(U)$ and $f \in L^{r}(U)$ with $r>\frac{d}{2}$, the boundedness and the continuity of a unique weak solution were verified in \cite{S65} based on De Giorgi's original ideas. In the general case of $c \geq 0$ in $U$ including the degenerate coefficients of $A=(a_{ij})_{1 \leq i,j \leq d}$, N.S. Trudinger showed in \cite{T73} (cf. \cite[Proposition 2.1.4]{BKRS15})
the existence and uniqueness of a weak solution to \eqref{maineq2} by using the well-posedness result for $c \geq \gamma$ in $U$ and the weak maximum principle (cf. \cite{T77} and \cite[Theorem 2.1.8]{BKRS15}) with a functional analytic result called Fredholm alternative (\cite[Theorem 6.6]{Br11}). \\
From the perspective of maintaining the continuity of the bilinear form associated with \eqref{maineq2}, the regularity conditions on $\bold{H}$ and $c$ mentioned in \cite{S65} seem optimal. However, by imposing a weak divergence-free condition on $\bold{H}$, even in the case where $\bold{H} \in L^2(U, \mathbb{R}^d)$, $A=id$ and $c=0$ in $U$, the existence and uniqueness of a bounded weak solution could be shown in \cite[Section 2.2]{Kon07}. 
Additionally, in the absence of singular zero-order terms, we refer to \cite{M11, KK15, KK19, HK21, HK22}  for results on the existence and uniqueness of the $W^{1,p}$-weak solutions of the divergence form equations with singular drifts.
When the singular zero-order terms $c$ and the drifts $\bold{H}$ belong to a Lorentz space, we mention \cite{GGM13, KT20, KO23, KPT23} for results on the existence and uniqueness of $W^{1,p}$-weak solutions of the divergence form equations. Here, we refer to \cite{KO23}, which presents results on the existence and uniqueness of very weak solutions closely related to invariant measures in stochastic analysis. When the zero-order terms $c$ belong to \( L^{\frac{pd}{p+d}}(U) \) with \( p \in (d, \infty) \) and the drifts $\bold{H}$ are in \( L^p(U, \mathbb{R}^d)\), refined Morrey estimates for weak solutions of the divergence form equations are obtained in \cite{S06}, assuming a priori the existence of a solution.
\\
Regarding more general conditions on the zero-order term, \cite{L24} demonstrates the existence and uniqueness of bounded weak solutions and $L^r$-contraction estimates of the solutions with $r \in [1, \infty]$ under the assumption that $\bold{H} \in L^{2}(U, \mathbb{R}^d)$ has negative weak divergence and $c \in L^{\frac{2d}{d+2}}(U)$ and $c \geq 0$ in $U$ if $d \geq 3$ and $c \in L^{\frac{r}{2}}(U)$ with $r>2$ and $c \geq 0$ in $U$ if $d=2$.
A natural question arises regarding the existence and uniqueness of a bounded weak solution when $c \in L^1(U)$ with $c \geq 0$ in $U$, and $\bold{H}$ no longer satisfies the negative weak divergence condition. Our first main result provides an answer to the natural question under the following conditions.\\ \\
{\bf (S)}:
{\it
$U$ is a bounded open subset of $\mathbb{R}^d$ with $d \geq 2$, $B_r(x_0)$ is an open ball in $\mathbb{R}^d$ with $\overline{U} \subset B_r(x_0)$, 
$\bold{H} \in L^p(U, \mathbb{R}^d)$, $h \in L^p(U)$  
with $p \in (d, \infty)$ satisfies $\| \bold{H}\| \leq h$ in $U$, and  $A=(a_{ij})_{1 \leq i,j \leq d}$ is a (possibly non-symmetric) matrix of measurable functions on $\mathbb{R}^d$ such that for some constants $M>0$ and $\lambda>0$ it holds that
\begin{equation} \label{elliptici}
\max_{1 \leq i,j \leq d} |a_{ij}(x)| \leq M, \;\; \; \; \langle A(x) \xi, \xi \rangle \geq \lambda \| \xi \|^2 \;\; \quad \text{ for a.e. $x \in \mathbb{R}^d$ and for all $\xi \in \mathbb{R}^d$}.
\end{equation}
}
\begin{theo} \label{mainexisthm}
Assume that {\bf (S)} holds, $c \in L^1(U)$ with $c \geq 0$ in $U$ and $f \in L^q(U)$ with $q \in (\frac{d}{2}, \infty)$. Then, the following hold:
\begin{itemize}
\item[(i)]
There exists $u \in H^{1,2}_0(U)_b$ (i.e. $u \in H^{1,2}_0(U) \cap L^{\infty}(U)$) such that $u$ is a weak solution to \eqref{maineq2},
 i.e. $u \in H^{1,2}_0(U)$ with $cu \in L^1(U)$ satisfies
\begin{equation} \label{weakdefzerve}
\int_U \langle A \nabla u, \nabla \psi \rangle + \langle \bold{H}, \nabla u \rangle \psi   + cu \psi dx = \int_U f \psi dx \quad \text{ for all } \psi \in C_0^{\infty}(U).
\end{equation}
Moreover, it holds that
\begin{equation} \label{estimenbd2} 
\|u\|_{H_0^{1,2}(U)}  \leq K_5 \|f\|_{L^{2_*}(U)}
\end{equation}
and that
\begin{equation} \label{bddestmizte}
\|u\|_{L^{\infty}(U)} \leq K_6 \|f\|_{L^q(U)},
\end{equation}
where $2_{*}:=\frac{2d}{d+2}$ if $d \geq 3$ and $2_{*}$ is an arbitrarily fixed number in $(1,2)$ if $d=2$, 
$K_5=K_1K_3$, $K_6=K_1K_4$, $K_1>0$ is a constant as in Theorem \ref{existrho} which depends only on $d$, $\lambda$, $M$, $r$, $p$, $\|h\|_{L^p(U)}$, $K_3>0$ is a constant which only depends on $d$, $\frac{\lambda}{K_1}$, $|U|$, and $K_4>0$ is a constant which only depends on $d$, $\frac{\lambda}{K_1}$, $|U|$ and $q$. 
\item[(ii)]
$u$ in Theorem \ref{mainexisthm}(i) is a unique weak solution to \eqref{maineq2}, i.e. if $v \in H^{1,2}_0(U)$ with $cv \in L^1(U)$ satisfies \eqref{weakdefzerve} where $u$ is replaced by $v$, then $u = v$ in $U$.
\item[(iii)]
If $\alpha>0$ is a constant, $c \geq \alpha$ in $U$ and $f \in L^{\theta}(U) \cap L^q(U)$ with $\theta \in [1, \infty]$, then $u$ in Theorem \ref{mainexisthm}(i) satisfies the following contraction estimates:
\begin{equation} \label{normalcontrac}
\|u \|_{L^{\theta}(U)} \leq \frac{K_1}{\alpha} \| f \|_{L^{\theta}(U)}.
\end{equation}
\end{itemize}
\end{theo}
The key ingredients of the proof of Theorem \ref{mainexisthm} are Theorems \ref{existrho}, \ref{mainequivthm}, where the main idea is to transform the equation \eqref{weakdefzerve} with general drift $\bold{H}$ into an equation with divergence-free drift.
We emphasize that the constants $K_1, K_3, K_4>0$ are independent of $c \in L^1(U)$. The contraction estimates in (iii) are remarkable in that $K_1 > 0$ is an explicit constant derived from the Harnack inequality (see the proof of Theorem \ref{existrho} and Remark \ref{imptremaforme}). Additionally, the contraction estimate with $\theta = 2$ is applicable in a posteriori error analysis (see Remark \ref{posteriana}).
Similar types of contraction estimates, as in Theorem \ref{mainexisthm}(iii) are also observed in  \cite[Chapter 11]{Kr08}, \cite[Theorem 3.4]{DL04} and \cite[Theorem 2.1]{KK19}. In case where $f$ is replaced by $f-{\rm div} \bold{F}$ in the weak sense, even in the restricted cases of $d = 2, 3$, results similar to Theorem \ref{mainexisthm} can be derived in Theorems \ref{mainexisthm2n} and \ref{specialresd}, respectively.  For more general high-dimensional cases where $f$ is replaced by $f-{\rm div} \bold{F}$ in the weak sense, the assumption that $a_{ij} \in VMO_{\omega}$ for all $1 \leq i, j \leq d$ is necessary to obtain results similar to Theorem \ref{mainexisthm} (see Theorem \ref{vmouniqweaksol}). \\
Next, in non-divergence type equation, let us investigate the previous results for the existence and uniqueness of a strong solution to \eqref{maineqnondiv}.
As a well-known result, \cite[Chapter 9]{GT01} establishes the existence and uniqueness of a strong solution to \eqref{maineqnondiv} under the assumption that $d \geq 2$, all components of $A$ are continuous in $\overline{U}$, $\widehat{\bold{H}} \in L^{\infty}(U, \mathbb{R}^d)$, and $c \in L^{\infty}(U)$ with $c \geq 0$ in $U$. In fact, under more general conditions allowing $\widehat{\bold{H}} \in L^d(U, \mathbb{R}^d)$ and $c \in L^{\frac{d}{2}}(U)$ with $d \geq 3$ and $\inf_{U} c >0$, the existence and uniqueness of a strong solution to \eqref{maineqnondiv} were already shown in \cite{C71}. However, for the general case of $c \geq 0$ in $U$, \cite{C71} assumes that an $i$th component of $\widehat{\bold{H}}$ is in $L^{\infty}(U)$ for some $1 \leq i \leq d$ to establish the well-posedness of strong solutions. 
By generalizing the continuity condition on the components of $A$ to the $VMO$ condition and considering the case where $\widehat{\bold{H}} = 0$ and $c = 0$, \cite{CFL93} proves the existence and uniqueness of a strong solution using the Calderón-Zygmund kernel.
In \cite{V92}, by exploiting the result in \cite{CFL93} and using a standard approximation argument and the Alexandrov-Pucci maximum principle, the existence and uniqueness of a strong solution to \eqref{maineqnondiv} are verified in case where $\widehat{\bold{H}} \in L^{r}(U, \mathbb{R}^d)$ and $c \in L^{d}(U)$ with $r>d \geq 3$.
Additionally in \cite{V93}, the same author further relaxed the condition on the coefficients in \cite{V92} to the one that $c \in L^{\frac{r}{2}}(U)$ with $r>d \geq 3$. More recently, N.V. Krylov established in \cite{K21} the existence and uniqueness of a strong solution to \eqref{maineqnondiv}, allowing 
$\widehat{\bold{H}} \in L^d(U, \mathbb{R}^d)$ and $c \in L^{\frac{d}{2}}(U)$ with $c \geq 0$ in $U$ if $d \geq 3$ and $c \in L^{s}(U)$ with $s>1$ and $c \geq 0$ in $U$ if $d=2$.
However, in situations where the number of data represented by 
$d$ is very large, the condition $c \in L^{\frac{d}{2}}(U)$ remains still restrictive. Our second main result demonstrates the existence and uniqueness of a strong solution to \eqref{maineqnondiv} for singular zero-order terms
and the precise statement is provided below.
\begin{theo} \label{seconmainres}
Let $U$ be a bounded open subset of $\mathbb{R}^d$ with $d \geq 2$ and $B_r(x_0)$ be an open ball in $\mathbb{R}^d$ with $\overline{U} \subset B_r(x_0)$. Assume that  $A=(a_{ij})_{1 \leq i,j \leq d}$ is a (possibly non-symmetric) matrix of measurable functions on $\mathbb{R}^d$ satisfying \eqref{elliptici}, $a_{ij} \in VMO_{\omega}$ for all $1 \leq i,j \leq d$, ${\rm div}A \in L^p(U, \mathbb{R}^d)$, and let $\widehat{\bold{H}} \in L^p(U, \mathbb{R}^d)$, $c \in L^s(U)$ with $c \geq 0$ in $U$ and $f \in L^q(U)$, where $p \in (d, \infty)$, $s \in (1, \infty)$ and $q \in (\frac{d}{2}, \infty)$.
Let $\beta:= p \wedge s  \wedge q$ and assume that $h \in L^p(U)$, $\tilde{c} \in L^s(U)$ satisfy $\|{\rm div}A + \widehat{\bold{H}} \| \leq h$ 
and $|c| \leq \tilde{c}$ in $U$. Additionally, assume that $\partial U$ is of class $C^{1,1}$.
Then, the following hold:
\begin{itemize}
\item[(i)]
There exists $u \in H^{2, \beta}(U) \cap H^{1,2}_0(U) \cap L^{\infty}(U)$ such that $u$ is a strong solution to \eqref{maineqnondiv}, i.e. $u \in H^{2,1}(U) \cap H^{1,2}_0(U)$ with $cu \in L^1(U)$ satisfies
\begin{equation} \label{strongformiden}
-\text{\rm trace}(A(x) \nabla^2 u(x)) + \langle \widehat{\bold{H}}(x), \nabla u(x) \rangle + c(x) u(x) =f(x), \quad \text{ for a.e. $x \in U$}.
\end{equation}
Moreover, $u$ satisfies the estimates \eqref{estimenbd2}, \eqref{bddestmizte}, and
$$
\|u\|_{H^{2,\beta}(U) } \leq K_{16} \|f\|_{L^q(U)},
$$
where $K_{16} > 0$ is a constant as in Theorem \ref{vmonondivh1p}(i) which depends only on $d$, $\lambda$, $M$, $r$, $p$, $s$, $q$, $\|h\|_{L^p(U)}$, $\|\tilde{c}\|_{L^s(U)}$ and $\omega$.
\item[(ii)]
$u$ in Theorem \ref{seconmainres}(i) is a unique strong solution to \eqref{maineqnondiv}, i.e. if $v \in H^{2,1}(U) \cap H^{1,2}_0(U)$ with $cv \in L^1(U)$ satisfies \eqref{strongformiden} where $u$ is replaced by $v$, then $u = v$ in $U$.
\item[(iii)]
If $\alpha>0$ is a constant, $c \geq \alpha$ in $U$ and $f \in L^{\theta}(U) \cap L^q(U)$ with $\theta \in [1, \infty]$, then
$u$ in Theorem \ref{seconmainres}(i) satisfies \eqref{normalcontrac}.
\end{itemize}
\end{theo}
\text{}\\
Our second main result relies heavily on our first main result, Theorem \ref{mainexisthm}, and Theorem \ref{vmonondivh1p}(i). Since the unique weak solution in Theorem \ref{mainexisthm} is in $L^{\infty}(U)$, the regularity condition on $c$ is allowed to be $c \in L^s(U)$ with any $s \in (1, \infty)$, which does not seem to be possible in \cite{K21} if $d \geq 3$. However, although the regularity of our zero-order term $c$ is quite low, the condition $f \in L^q(U)$ with $q \in (\frac{d}{2}, \infty)$ is essential in Theorem \ref{seconmainres} to ensure the existence of a bounded weak solution $u$ as in Theorem \ref{mainexisthm}. On the other hand, although restrictive conditions such as $c \in L^{\frac{d}{2}}(U)$ appear in \cite{K21}, $f$ is allowed to belong to $L^r(U)$ with arbitrary $r \in (1, \infty)$. It remains an open question whether our second main result still holds when the conditions $\text{div} A ,\, \bold{H} \in L^p(U, \mathbb{R}^d)$ are relaxed. \\
Finally, let us discuss the novelty of this paper, particularly regarding the constants that appear in our estimates \eqref{estimenbd2}, \eqref{bddestmizte}, and \eqref{normalcontrac}. For instance, under certain regularity assumptions on the coefficients in \eqref{maineq2}, as a consequence of the bounded inverse theorem (see \cite{Br11}), one may show the existence of a constant $\tilde{C}>0$, independent of $f$ and $u$ satisfying that
\begin{equation} \label{bddinvl2esti}
\|u\|_{L^2(U)} \leq \hat{C}\|f\|_{L^2(U)},
\end{equation}
(see  \cite[Section 6.2, Theorem 6]{Ev10}, \cite[Lemma 9.17]{GT01} and \cite[Proof of Theorem 2.5]{HK22}). However, though we can confirm that the constant $\hat{C}>0$ is independent of $f$ and $u$, it is difficult to determine precisely how it depends on the quantities in the equations \eqref{maineq2} and \eqref{maineqnondiv}. For instance, if $\overline{\bold{H}}$ and $\overline{c}$ are sufficiently close to $\bold{H}$ and $c$ in some sense, and $\bold{H}$ and $c$ in the equation \eqref{maineq2}  are replaced by $\overline{\bold{H}}$ and $\overline{c}$, respectively,  it may not be easy to ascertain how the constant $\hat{C}>0$ in \eqref{bddinvl2esti} robustly maintains its value. In contrast, the constants presented in our estimates \eqref{estimenbd2}, \eqref{bddestmizte}, and \eqref{normalcontrac} reveal their concrete dependence on the quantities in the equations \eqref{maineq2} and \eqref{maineqnondiv}, and indeed, these constants can be calculated explicitly.  Furthermore, the constant that appears in the contraction estimates \eqref{normalcontrac} is of the form $\frac{K_1}{\alpha}$, where $K_1>0$ is indeed an upper bound of $\frac{\max_{\overline{U}} \rho}{\min_{\overline{U}} \rho}$ (cf. Remark \ref{imptremaforme}), allowing for the computation of a specific value. From this perspective, our result is distinguished from others in that the constants appearing in the estimates can be concretely calculated. Even when the regularity of the coefficients is very high, our results offer new insights into the explicit constants in the estimates, which can be applied to a posteriori error analysis  (see Remark \ref{posteriana}). \\
This paper is organized as follows: Section \ref{notconv} introduces the primary notations and definitions used throughout this paper. Section \ref{maindivform} presents the fundamental results of this paper, Theorems \ref{existrho}, \ref{mainequivthm}, which lead us to establish the existence and uniqueness of a weak solution to \eqref{maineq2} when the components of $A$ are merely bounded measurable functions. The proof of the first main result is provided after Remark \ref{imptremau}. Section \ref{vmow1pdivse} investigates the existence and uniqueness of a weak solution to \eqref{maineq2} when the components of $A$ belong to the $VMO$ class, and $f$ is replaced by $f-{\rm div}\,\bold{F}$ in the weak sense. In Section \ref{nondivresec}, we prove our second main result after Remark \ref{slightremar} by verifying in Theorem \ref{vmonondivh1p}(i) that the unique weak solution constructed in Theorem \ref{mainexisthm} is indeed twice weakly differentiable under the additional conditions. Appendix \ref{auxilarlse} introduces auxiliary results that are essential to the main content of this paper.

\section{Notations and definitions} \label{notconv}
In this paper, we treat the Euclidean space $\mathbb{R}^d$, equipped with the Euclidean
inner product $\langle \cdot,  \cdot \rangle$ and the Euclidean norm $\| \cdot \|$. For each $x_0 \in \mathbb{R}^d$ and $r>0$, let $B_r(x_0):= \{ x \in \mathbb{R}^d : \|x-x_0 \|<r \}$.
We write $a \wedge b$ = $b \wedge a:= a$ and $a \vee b= b \vee a :=b$ for $a, b \in \mathbb{R}$ with $a \leq b$. For $W \subset \mathbb{R}^d$, $1_{W}$ denotes the indicator function of $W$. The Lebesgue measure is denoted by $dx$ and we write $|E|:=dx(E)$, where $E$ is a Lebesgue measurable set in $\mathbb{R}^d$. Let $U$ be an open subset of $\mathbb{R}^d$. $\mathcal{B}(U)$ denotes the set of all Borel measurable functions on $U$. 
For $\mathcal{A} \subset \mathcal{B}(U)$, $\mathcal{A}_0$ denotes the set of all functions $f \in \mathcal{A}$ such that  $\text{supp}(f\,dx)$ is compact in $U$.
The sets of all continuous functions on $U$ and $\overline{U}$ are denoted by $C(U)$ and $C(\overline{U})$, respectively, and define $C_0(U):=C(U)_0$. Let $k \in \mathbb{N} \cup \{ \infty \}$. The set of all $k$-times continuously differentiable functions is denoted by $C^k(U)$ and define $C^k_0(U):=C^k(U) \cap C_0(U)$. Denote by $C^k(\overline{U})$ the set of all continuous functions $f$ on $\overline{U}$ satisfying that  there exist an open subset $V$ of $\mathbb{R}^d$ with $\overline{U} \subset V$ and $\tilde{f} \in C^k(V)$ such that $\tilde{f}=f$ in $\overline{U}$. Let $r \in [1, \infty]$. Denote by $L^r(U)$ the $L^r$-space on $U$ with respect to $dx$ equipped with usual $L^r$-norm on $U$. $L^r(U, \mathbb{R}^d)$ denotes the space of all $L^r$-vector fields on $U$ with respect to $dx$ equipped with the norm $\| \bold{F} \|_{L^r(U)}:=\big\| \| \bold{F}\| \big\|_{L^r(U)}$. $L^r_{loc}(U)$ denotes the set of all Borel measurable functions $f$ on $U$ satisfying that, $f 1_{W} \in L^r(U)$ for any bounded open subset $W$ of $\mathbb{R}^d$ with $\overline{W} \subset U$. Likewise, denote $L_{loc}^r(U, \mathbb{R}^d)$ the set of all Borel measurable vector fields $\bold{F}$ on $U$ satisfying $\| \bold{F} \| \in L^r_{loc}(U)$. For each $i=1, \ldots, d$ and a function $f$ on $U$,  let $\partial_i f$ be the $i$-th weak partial derivative of $f$ on $U$, if it exists. The Sobolev space $H^{1,r}(U)$ is defined to be the space of all functions $f \in L^r(U)$ for which $\partial_i f \in L^r(U)$ for all $i=1, \ldots d$ equipped with the usual $H^{1,r}(U)$-norm. For each $q \in [1, \infty)$,
$H^{1,q}_0(U)$ denotes the closure of $C_0^{\infty}(U)$ in $H^{1,q}(U)$. Specifically, we write $H_0^{1,2}(U)_b := H^{1,2}_0(U) \cap L^{\infty}(U)$. The Sobolev space $H^{2,r}(U)$ is defined to be the space of all functions $f \in L^r(U)$ for which $\partial_i f \in L^r(U)$ and $\partial_{i} \partial_j f \in L^r(U)$ for each $i, j =1, \ldots d$ equipped with the usual $H^{2,r}(U)$-norm. $\Delta$ denotes the weak Laplacian, defined as $\Delta f:= \sum_{i=1}^d \partial_i \partial_i f$. For a twice weakly differentiable function $f$, $\nabla^2 f$ denotes a weak Hessian matrix. i.e. $\nabla^2 f:= (\partial_i \partial_j f)_{1\leq i,j \leq d}$. For a matrix $B=(b_{ij})_{1 \leq i,j \leq d}$, we write $\text{trace}(B):=\sum_{i=1}^d b_{ii}$. In this paper, $2_{*}$ denotes an arbitrary but fixed number in $(1,2)$ if $d=2$ and $2_{*}:=\frac{2d}{d+2}$ if $d \geq 3$.
\begin{defn}
Let $U$ be an open subset of $\mathbb{R}^d$ with $d\geq 2$. We say that the boundary $\partial U$ is of class $C^1$ (resp. $C^{1,1}$) if for each $x_0 \in \partial U$ there exists $r>0$ and a $C^1$ (resp. $C^{1,1}$)-function $\sigma: \mathbb{R}^{d-1} \rightarrow \mathbb{R}$ such that (under rotating and
relabeling the coordinates axes if necessary) we have
$$
U \cap B_{r}(x_0) = \{ x \in B_r(x_0): x_d >  \sigma(x_1, \ldots, x_{d-1})	\}.
$$
\end{defn}
\centerline{}
\begin{defn} \label{divecdefn}
Let $U$ be an open subset of $\mathbb{R}^d$, $\bold{E}=(e_1, \ldots, e_d) \in L^1_{loc}(U, \mathbb{R}^d)$ and $B=(b_{ij})_{1 \leq i,j \leq d}$ be a (possibly non-symmetric) matrix of functions in $L^1_{loc}(U)$. 
We write ${\rm div} B = \bold{E}$ if
$$
\int_{U} \sum_{i,j=1}^d b_{ij} \partial_i \phi_j\, dx = - \int_{U} \sum_{j=1}^d e_{j} \phi_j dx, \quad \text{ for all } \phi_j \in C_0^{\infty}(U),\;  j=1, \ldots, d.
$$
\end{defn}
\centerline{}
\begin{defn} \label{vmodefn}
Let $\omega$ be a positive continuous function on $[0, \infty)$ with $\omega(0)=0$. $VMO_{\omega}$ denotes the set of all functions $g \in L^1_{loc}(\mathbb{R}^d)$ satisfying that,
$$
\sup_{z \in \mathbb{R}^d, r<R} r^{-2d}\int_{B_r(z)} \int_{B_r(z)} |g(x)-g(y)| dx dy  \leq \omega(R), \quad \text{ for all $R \in (0, \infty)$}.
$$
\end{defn}
\centerline{}
\begin{defn} \label{defnofweaksol}
Let $U$ be a bounded open subset of $\mathbb{R}^d$ and $A=(a_{ij})_{1 \leq i,j \leq d}$ be a (possibly non-symmetric) matrix of bounded and measurable functions on $\mathbb{R}^d$ satisfying \eqref{elliptici}. For $\bold{H} \in L^2(U, \mathbb{R}^d)$, $c \in L^1(U)$, 
$f \in L^1(U)$ and $\bold{F} \in L^2(U, \mathbb{R}^d)$, we say that $u$ is a weak solution to 
\begin{equation} \label{maineq}
\left\{
\begin{alignedat}{2}
-\text{\rm div}(A \nabla u) + \langle \bold{H}, \nabla u \rangle + c u &=f - \text{\rm div} \bold{F} && \quad \mbox{{\rm in}\; $U$}\\
u &= 0 &&\quad \mbox{{\rm on}\; $\partial U$},
\end{alignedat} \right.
\end{equation}
 if $u \in H^{1,2}_0(U)$ with $cu \in L^1(U)$
satisfies
\begin{equation*}
\int_U \langle A \nabla u, \nabla \psi \rangle + \langle \bold{H}, \nabla u \rangle \psi   + cu \psi dx = \int_U f \psi +\langle \bold{F}, \nabla \psi \rangle dx, \quad \text{ for all } \psi \in C_0^{\infty}(U).
\end{equation*}
\end{defn}
\centerline{}
\begin{defn} \label{defnstrongsol}
Let $U$ be a bounded open subset of $\mathbb{R}^d$, $A=(a_{ij})_{1 \leq i,j \leq d}$ be a (possibly non-symmetric) matrix of bounded and measurable functions on $\mathbb{R}^d$ satisfying \eqref{elliptici}.
Let \,$\widehat{\bold{H}}=(\hat{h}_1, \ldots, \hat{h}_d) \in L^2(U, \mathbb{R}^d)$, $c \in L^1(U)$ and $f \in L^1(U)$. Here, $u$ is called a strong solution to \eqref{maineqnondiv}
if $u \in H^{2, 1} (U)\cap H^{1,2}_0(U)$ with $cu \in L^1(U)$ satisfies \eqref{strongformiden},
i.e.
$$
-\sum_{i,j=1}^d a_{ij}(x) \partial_{i} \partial_j u(x) + \sum_{i=1}^d \hat{h}_i(x) \partial_i u(x) + c(x) u(x)  = f(x), \quad \text{ for a.e. $x \in U$}.
$$
\end{defn}

\section{Divergence type equations with measurable $A$} \label{maindivform}
\subsection{The case of $d \geq 2$, $c \in L^1(U)$, $f \in L^{q}(U)$ with $q \in (\frac{d}{2}, \infty)$ and  $\bold{F}= 0$ }

\noindent
Theorem \ref{existrho} presented below are key components of this paper and are essential for transforming the drift vector field $\bold{H}$ in \eqref{maineq} into the divergence-free vector field $\rho \bold{B}$ in Theorem \ref{mainequivthm}. The primary approach underlying these results is based on the weak maximum principle in \cite[Theorem 4]{T77} and the Harnack inequality in \cite[Theorem 4]{AS67} (cf. \cite[Corollary 5.5]{T73}). The main idea of the proof was first introduced in \cite[Theorem 1]{BRS12}  (cf. \cite[Theorem 2.2.1]{BKRS15}) and later generalized in \cite[Theorem 2.27(i)]{LST22}.
\begin{theo} \label{existrho}
Assume that {\bf (S)} holds. Then, the following hold:
\begin{itemize}
\item[(i)]
Let $x_1 \in U$. Then, there exists $\rho \in H^{1,2}(B_{4r}(x_0)) \cap C(B_{4r}(x_0))$ with $\rho(x)>0$ for all $x \in B_{4r}(x_0)$ and $\rho(x_1)=1$ such that
\begin{equation} \label{infinva}
\int_{B_{4r}(x_0)} \langle A^T \nabla \rho  + \rho \bold{H}, \nabla \varphi \rangle dx = 0, \quad \text{ for all $\varphi \in C_0^{\infty}(B_{4r}(x_0))$}.
\end{equation}
\item[(ii)]
Let $\rho$ be as in Theorem \ref{existrho}(i). Then, there exists a constant $K_1 \geq 1$ which only depend on $d$, $\lambda$, $M$, $r$, $p$, $\|h\|_{L^p(U)}$ such that \;$\max_{\overline{B}_{3r}(x_0)} \rho  \leq K_1 \min_{\overline{B}_{3r}(x_0)} \rho$, and hence
$$
1 \leq \max_{\overline{U}} \rho  \leq K_1 \min_{\overline{U}} \rho \leq K_1.
$$
\item[(iii)]
Let $\rho$ be as in Theorem \ref{existrho}(i) and $K_1 \geq 1$ be a constant as in Theorem \ref{existrho}(ii). Then, it holds that
$$
\| \nabla \rho \|_{L^2(B_{2r}(x_0))} \leq K_1 K_2,
$$
where 
$\displaystyle K_2:= \left( \frac{8d^2M^2}{\lambda^2 r^2} |B_{3r}(x_0)|	 + \frac{2}{\lambda^2} |B_{3r}(x_0)|^{1-\frac{2}{p}} \|h\|^2_{L^p(U)}  + \frac{2}{r \lambda} |B_{3r}(x_0)|^{1-\frac{1}{p}} \|h\|_{L^p(U)} \right)^{1/2}$.
\end{itemize}
\end{theo}
\begin{proof}
(i) Extend $\bold{H}$ and $h$ on $\mathbb{R}^d$ by the zero-extension, and say again $\bold{H}$ and $h$, respectively.
By \cite[Lemma 2.25]{LST22} and \cite[Corollary 5.5]{T73} there exists $v \in H^{1,2}_0(B_{4r}(x_0)) \cap C(B_{4r}(x_0))$ such that
$$
\int_{B_{4r}(x_0)} \langle A^T \nabla v + v \bold{H}, \nabla \varphi   \rangle dx = \int_{B_{4r}(x_0)} \langle - \bold{H}, \nabla \varphi \rangle dx, \quad \text{ for all $\varphi \in C_0^{\infty}(B_{4r}(x_0))$.}
$$
Let $w:=v+1 \in H^{1,2}(B_{4r}(x_0)) \cap C(B_{4r}(x_0))$. Let $T: H^{1,2}(B_{4r}(x_0)) \rightarrow L^2(\partial B_{4r}(x_0))$ be the trace operator defined as in \cite[Theorem 4.6]{EG15}.
Then, we get
\begin{equation} \label{traceres}
T(w)=1 \; \text{ in } \; L^2(\partial B_{4r}(x_0)),
\end{equation}
since $T$ is a linear operator and $T(v)=0$ in  $L^2(\partial B_{4r}(x_0))$.
Now, we have
\begin{equation} \label{originexist}
\int_{B_{4r}(x_0)} \langle A^T \nabla w+ w \bold{H}, \nabla \varphi   \rangle dx = 0, \quad \text{ for all $\varphi \in C_0^{\infty}(B_{4r}(x_0))$}.
\end{equation}
Observe that $0 \leq w^- \leq v^-$ and $v^- \in H^{1,2}_0(B_{4r}(x_0))$, so that $w^- \in H^{1,2}_0(B_{4r}(x_0))$ by Proposition \ref{sandwich}.  Moreover, by \cite[Lemma 3.4]{LT21}, we have
$$
\int_{B_{4r}(x_0)} \langle A^T \nabla w^- + w^- \, \bold{H}, \nabla \varphi   \rangle dx \leq 0, \quad \text{ for all $\varphi \in C_0^{\infty}(B_{4r}(x_0))$ \,with\, $\varphi \geq 0$}.
$$
Using the weak maximum principle, \cite[Theorem 4]{T77} (cf. \cite[Theorem 2.1.8]{BKRS15}), we get $w^- \leq 0$ in $B_{4r}(x_0)$ and hence $w \geq 0$ in $B_{4r}(x_0)$. Here, we emphasize the fact that $w \in C(B_{4r}(x_0))$. Now we claim that $w(x)>0$ for all $x \in B_{4r}(x_0)$. To show the claim, use the proof by contradiction.
Suppose that $w(y_0)=0$ for some $y_0 \in B_{4r}(x_0)$. Then, by the Harnack inequality \cite[Corollary 5.5]{T73},  
one can easily show that $w=0$ on $B_R(x_0)$ for any $R \in (\|y_0-x_0\|, 4r)$. Since $R$ is arbitrary, $w=0$ on $B_{4r}(x_0)$. Thus, $T(w)=0$ in $L^2(\partial B_{4r}(x_0))$, which contradicts to \eqref{traceres}. Thus, we conclude that $w(x)>0$ for all $x \in B_{4r}(x_0)$.  Since $w(x_1)>0$, we can define
$$
\rho(x):= \frac{1}{w(x_1)} w(x), \quad x \in B_{4r}(x_0).
$$
Then, (i) follows.
\\ \\
(ii) Using the Harnack inequality, \cite[Theorem 4]{AS67} 
(Note:  \cite[Theorem 4]{AS67} is a parabolic Harnack inequality, but is still applicable to our elliptic problem \eqref{originexist}. The reason we use \cite[Theorem 4]{AS67} is that the constant provided therein can be identified more explicitly than in \cite[Corollary 5.5]{T73}. Specifically, we want the constant \( K_1 > 0 \) to depend on \( \|h\|_{L^p(U)} \) rather than \( \|\bold{H}\|_{L^p(U)} \), as this plays a crucial role in performing approximation arguments.), there exists $K_1>0$ which only depends on $d$, $\lambda, M, r, p$ and $\|h\|_{L^p(U)}$ such that
\begin{equation} \label{harnaell}
\max_{\overline{B}_{3r}(x_0)} \rho \leq K_1 \min_{\overline{B}_{3r}(x_0)} \rho.
\end{equation}
Since $\max_{\overline{U}} \rho \leq \max_{\overline{B}_{3r}(x_0)} \rho$ and $\min_{\overline{B}_{3r}(x_0)} \rho 
\leq \min_{\overline{U}} \rho$, the assertion follows. \\ \\
(iii)  Let $\chi \in  H^{1,\infty}(\mathbb{R}^d)_0 \cap H^{1,2}_0(B_{3r}(x_0))$ be defined by
\begin{equation*} \label{bumpftn}
\chi(x) := \min \left( \left(3-\frac{\|x-x_0\|}{r}\right)^+, \;1 \right), \quad x \in \mathbb{R}^d.
\end{equation*}
Then, it follows that $\chi(x)=1$ for all $x \in \overline{B}_{2r}(x_0)$, $|\chi(x)| \leq 1$ and $\| \nabla \chi(x) \| \leq \frac{1}{r}$ for a.e. $x  \in \mathbb{R}^d$.
Replacing $\varphi$ in \eqref{originexist} by $\chi^2 \rho$, we hence get
$$
\int_{B_{3r}(x_0)} \langle A^T \nabla \rho + \rho \bold{H}, \, \chi^2 \nabla \rho + 2\rho \chi \nabla \chi \rangle \,dx = 0,
$$
so that
$$
\int_{B_{3r}(x_0)} \langle \chi^2 A^T \nabla \rho, \nabla \rho \rangle\, dx = - \int_{B_{3r}(x_0)} \langle A^T \nabla \rho, 2\rho \chi \nabla \chi \rangle dx  -\int_{B_{3r}(x_0)}  \langle \rho \bold{H}, \chi^2 \nabla \rho \rangle dx -\int_{B_{3r}(x_0)} \langle \rho \bold{H}, 2\rho \chi \nabla \chi \rangle dx.
$$
Thus, \eqref{elliptici} and the above yield that
\begin{align*}
\int_{B_{3r}(x_0)} \lambda \chi^2 \| \nabla  \rho \|^2 dx &\leq \int_{B_{3r}(x_0)} dM \| \chi\nabla \rho \| \cdot \| 2\rho  \nabla \chi \| dx  + \int_{B_{3r}(x_0)} \| \chi \nabla \rho \| \cdot \| \chi \rho \bold{H} \| dx \\
& \qquad + \int_{B_{3r}(x_0)} \rho^2 \|\bold{H}\| \cdot \| \chi \nabla \chi \| dx \\
& \leq \int_{B_{3r}(x_0)} \| \chi \nabla \rho \| \left(  \frac{2dM}{r}	 + h \right) |\rho| dx + \frac{1}{r}\int_{B_{3r}(x_0)} \rho^2 h\, dx \\
& \leq \int_{B_{3r}(x_0)}\frac{\lambda}{2} \chi^2 \| \nabla \rho \|^2 dx  + \int_{B_{3r}(x_0)} \left( \frac{1}{2\lambda} \left(  \frac{2dM}{r}	 + h \right)^2 + \frac{1}{r} h  \right) \rho^2  dx.
\end{align*}
Therefore,  \eqref{harnaell} implies that
\begin{align*}
\int_{U} \|\nabla \rho \|^2 dx &\leq \int_{B_{3r}(x_0)} \chi^2 \| \nabla \rho \|^2 dx \\
& \leq \int_{B_{3r}(x_0)} \left( \frac{1}{\lambda^2} \left(  \frac{2dM}{r}	 +h \right)^2 + \frac{2}{r \lambda} h  \right) \rho^2  dx \\
& \leq \left( \max_{B_{3r}(x_0)} \rho \right)^2   \int_{B_{3r}(x_0)} \frac{1}{\lambda^2}  \left(  \frac{2dM}{r}	 + h \right)^2 + \frac{2}{r \lambda} h   dx \\
& \leq K_1^2 \left( \min_{B_{3r}(x_0)} \rho \right)^2 \int_{B_{3r}(x_0)} \frac{1}{\lambda^2} \left(  \frac{2dM}{r}	 + h \right)^2 + \frac{2}{r \lambda} h  dx \\
& \leq K_1^2  K_2^2,
\end{align*}
as desired.
\end{proof}

\begin{theo} \label{mainequivthm}
Assume that {\bf (S)} holds.
Let $\rho \in H^{1,2}(U) \cap C(\overline{U})$ be a strictly positive function on $\overline{U}$ constructed as in Theorem \ref{existrho}.
Define 
$$
\bold{B}:=\bold{H}+\frac{1}{\rho} A^T \nabla \rho \quad \text{on $U$}.
$$
Then, $\rho \bold{B} \in L^2(U, \mathbb{R}^d)$ and 
$$
\int_{U} \langle \rho \bold{B}, \nabla \varphi \rangle dx = 0, \quad \text{ for all $\varphi \in C_0^{\infty}(U)$}.
$$
Let $f \in L^1(U)$, $\bold{F} \in L^2(U, \mathbb{R}^d)$ and $u \in H^{1,2}_0(U)$ with $cu \in L^1(U)$. Then the following (i) and (ii) are equivalent.
\begin{itemize}
\item[(i)]
\begin{equation} \label{maineqfir}
\int_U \langle A \nabla u, \nabla \psi \rangle +\langle \bold{H}, \nabla u \rangle \psi  + cu \psi dx = \int_U f \psi + \langle \bold{F}, \nabla \psi \rangle dx, \quad \text{ for all } \psi \in C_0^{\infty}(U).
\end{equation}
\item[(ii)]
\begin{equation} \label{maineqfdiv}
\int_U \langle \rho A \nabla u, \nabla \varphi \rangle   +  \langle \rho \bold{B}, \nabla u \rangle \varphi   + \rho c u \varphi dx = \int_U f \rho \varphi   + \langle \bold{F}, \nabla (\rho \varphi) \rangle dx, \quad \text{ for all } \varphi \in C_0^{\infty}(U).
\end{equation}
\end{itemize}
\end{theo}
\begin{proof}
(i) $\Rightarrow$ (ii). By an approximation as in Proposition \ref{propapprobdd}, \eqref{maineqfir} holds where \lq\lq for all $\psi \in C_0^{\infty}(\mathbb{R}^d)$\rq\rq \; is replaced by 
 \lq\lq for all $\psi \in H^{1,2}_0(U)_b$\rq\rq. 
 Let $\varphi \in C_0^{\infty}(U)$ be arbitrarily given. Replacing $\psi$ by $\rho \varphi \in H^{1,2}_0(U)_b$ in \eqref{maineqfir}, we discover that
 \begin{align*}
&\int_U f \varphi \rho  + \langle \bold{F}, \nabla (\rho \psi) \rangle dx = \int_U \langle A \nabla u, \nabla (\rho \varphi) \rangle  +  \langle \bold{H}, \nabla u \rangle \rho \varphi    +  cu \rho \varphi dx \\
 &= \int_{U} \langle \rho A \nabla u, \nabla  \varphi \rangle   + \left \langle \frac{1}{\rho}A^T \nabla \rho + \bold{H}, \nabla u \right \rangle \rho \varphi   + \rho cu \varphi  dx \\
 &= \int_U \langle \rho A \nabla u, \nabla \varphi \rangle + \langle \rho \bold{B}, \nabla u \rangle \varphi +  \rho cu \varphi dx.
 \end{align*}
(ii) $\Rightarrow$ (i).
By an approximation as in Proposition \ref{propapprobdd}, \eqref{maineqfdiv} holds where \lq\lq for all $\varphi \in C_0^{\infty}(\mathbb{R}^d)$\rq\rq\;is replaced by 
 \lq\lq for all $\varphi \in H^{1,2}_0(U)_b$\rq\rq. Let $\psi \in C_0^{\infty}(U)$ be arbitrarily given. Replacing $\varphi$ by $\frac{\psi}{\rho} \in H^{1,2}_0(U)_b$ in \eqref{maineqfdiv}, we discover that
 \begin{align*}
 &\int_{U} f \psi  +  \langle \bold{F}, \nabla \psi \rangle dx = \int_U \left \langle \rho A \nabla u, \nabla \left(  \frac{\psi}{\rho} \right) \right\rangle  + \langle \bold{B}, \nabla u \rangle \psi   +cu \psi dx\\
 & = \int_{U} \langle A \nabla u, \nabla \psi \rangle dx + \int_{U} \left \langle -\frac{1}{\rho}A^T \nabla \rho + \bold{B}, \nabla u \right \rangle  \psi + cu \psi dx  \\
 & = \int_U \langle A \nabla u, \nabla \psi \rangle + \langle \bold{H}, \nabla u \rangle \psi  +  cu \psi dx.
 \end{align*}
\end{proof}

\begin{theo} \label{theomaineun}
Let $\hat{\lambda}>0$ be a constant and $\hat{A}=(\hat{a}_{ij})_{1 \leq i,j \leq d}$ be a matrix of bounded and measurable functions on $\mathbb{R}^d$ such that
\begin{equation*} \label{ellipticiahat}
\langle \hat{A}(x) \xi, \xi \rangle \geq \hat{\lambda} \| \xi \|^2, \quad \text{for a.e. $x \in \mathbb{R}^d$ and} \text{ for all } \xi \in \mathbb{R}^d.
\end{equation*}
Let $\hat{\bold{B}} \in L^2(U, \mathbb{R}^d)$ be a vector field such that
\begin{equation} \label{weakdivfree} 
\int_{U} \langle \hat{\bold{B}}, \nabla \varphi \rangle dx = 0, \quad \text{ for all $\varphi \in C_0^{\infty}(U)$}.
\end{equation}
Let $\hat{c} \in L^1(U)$ with $\hat{c} \geq 0$ in $U$, $\hat{f} \in L^{q}(U)$ and $\hat{\bold{F}} \in L^{2q}(U, \mathbb{R}^d)$ with $q \in (\frac{d}{2}, \infty)$.
Then, the following hold:
\begin{itemize}
\item[(i)]
There exists a weak solution $\hat{u} \in H^{1,2}_0(U)_b$ to
\begin{equation} \label{mainenerbd}
\left\{
\begin{alignedat}{2}
-{\rm div}(\hat{A} \nabla \hat{u}) + \langle \hat{\bold{B}}, \nabla \hat{u} \rangle + \hat{c} \hat{u} &=\hat{f} - {\rm div} \hat{\bold{F}} && \quad \mbox{\;{\rm in} \; $U$}\\
\hat{u} &= 0 &&\quad \mbox{\;{\rm on}\; $\partial U$},
\end{alignedat} \right.
\end{equation}
i.e. $\hat{u} \in H^{1,2}_0(U)$ with $\hat{c} \hat{u} \in L^1(U)$ satisfies
\begin{equation} \label{zeroforc}
\int_U \langle  \hat{A} \nabla \hat{u}, \nabla \varphi \rangle  + \langle \hat{\bold{B}}, \nabla \hat{u} \rangle \varphi  +  \hat{c}\hat{u} \varphi dx  = \int_{U} \hat{f} \varphi  dx + \int_{U} \langle \hat{\bold{F}}, \nabla \varphi \rangle dx, \quad \text{ for all $\varphi \in C_0^{\infty}(U)$}.
\end{equation}
Moreover, it holds that
\begin{align} 
\|\hat{u}\|_{H_0^{1,2}(U)}  &\leq  \hat{K}_3 \left( \|f\|_{L^{2_*}(U)} + \|\hat{\bold{F}} \|_{L^2(U)}	\right),	\label{estimenbd}  \\	
\qquad \|\hat{u}\|_{L^{\infty}(U)} &\leq  \hat{K}_4 \left( \|f\|_{L^q(U)} + \| \hat{\bold{F}}\|_{L^{2q}(U)} \right), 	\label{bddestimdivf}
\end{align}
where 
$\hat{K}_3 >0$ is a constant which only depends on $d$, $\hat{\lambda}$ and $|U|$, and $\hat{K}_4>0$ is a constant which only depends on $d$, $\hat{\lambda}$, $q$ and $|U|$.
\item[(ii)] Let $\hat{v} \in H_0^{1,2}(U)$ be such that $\hat{c} \hat{v} \in L^1(U)$ and that
\begin{equation} \label{zeroforcdu}
\int_U \langle \hat{A} \nabla \hat{v}, \nabla \varphi \rangle +  \langle \hat{\bold{B}}, \nabla \hat{v} \rangle \varphi  + \hat{c} \hat{v} \varphi dx  = 0, \quad \text{ for all $\varphi \in C_0^{\infty}(U)$}.
\end{equation}
Then, $\hat{v}=0$ in $U$, which concludes from the linearity that $\hat{u}$ in Theorem \ref{theomaineun}(i) is a unique weak solution to \eqref{mainenerbd}.
\item[(iii)]
Let $\alpha>0$ and $\theta \in [1, \infty]$ be constants and assume that $\hat{c}\geq \alpha$, $\hat{f} \in L^{\theta}(U) \cap L^q(U)$. Then, $\hat{u}$ in Theorem \ref{theomaineun}(i) satisfies
$$
\|\hat{u} \|_{L^{\theta}(U)} \leq \frac{1}{\alpha} \| \hat{f} \|_{L^{\theta}(U)} +|U|^{1/2} \hat{K}_4 \| \bold{F} \|_{L^{2q}(U)},
$$
where $\hat{K}_4>0$ is a constant as in Theorem \ref{theomaineun}(i).
\end{itemize}
\end{theo}
\begin{proof}
(i) For each $n \geq 1$, let $\hat{c}_n:= \hat{c} \wedge n$. Then, $(\hat{c}_n)_{n \geq 1}$  is a sequence of functions in $L^{\infty}(U)$ with $c_n \geq 0$ for all $n \geq 1$ such that $\lim_{n \rightarrow \infty} \hat{c}_n = \hat{c}$ a.e. in $U$ aud in $L^1(U)$. 
By \cite[Theorem 1.1(i)]{L24}, there exists $\hat{u}_n \in H^{1,2}_0(U)_b$, $n \geq 1$,  such that
\begin{equation} \label{ourweigheq}
\int_U \langle \hat{A} \nabla \hat{u}_n, \nabla \varphi \rangle  +  \langle \hat{\bold{B}}, \nabla \hat{u}_n \rangle \varphi  + \hat{c}_n \hat{u}_n \varphi dx  = \int_{U} \hat{f} \varphi  dx + \int_{U} \langle \hat{\bold{F}}, \nabla \varphi \rangle dx  \quad \text{ for all $\varphi \in C_0^{\infty}(U)$}
\end{equation}
and that
\begin{equation}  \label{energes}
\|\hat{u}_n\|_{H_0^{1,2}(U)}  \leq \hat{K}_3 \left(\| \hat{f}\|_{L^{2_*}(U)}  + \| \hat{\bold{F}}\|_{L^{2}(U)} \right),
\end{equation}
\begin{equation} \label{bddesti}
\|\hat{u}_n\|_{L^{\infty}(U)} \leq \hat{K}_4 \left( \|\hat{f}\|_{L^q(U)} + \| \hat{\bold{F}}\|_{L^{2q}(U)}  \right),
\end{equation}
where $\hat{K}_3, \hat{K}_4>0$ are constants as in (i).
Therefore, using \eqref{energes} and \eqref{bddesti} with the weak compactness of $H^{1,2}_0(U)$,
there exist $\hat{u} \in H^{1,2}_0(U)_b$ and a subsequence of $(\hat{u}_n)_{n \geq 1}$, say again $(\hat{u}_n)_{n \geq 1}$, such that
\begin{equation} \label{weakptconv}
\lim_{n \rightarrow \infty} \hat{u}_n = \hat{u} \text{ \; weakly in\; $H^{1,2}_0(U)$}, \;\;\; \lim_{n \rightarrow \infty} \hat{u}_n = \hat{u}  \text{ a.e. on $U$}.
\end{equation}
Letting $n \rightarrow \infty$ in \eqref{ourweigheq}, \eqref{energes} and \eqref{bddesti} and using \eqref{weakptconv} and \eqref{bddesti}, we get \eqref{zeroforc}, \eqref{estimenbd} and \eqref{bddestimdivf}.\\ \\
\noindent
(ii) 
Let $\phi \in C_0^{\infty}(U)$ be arbitrarily given. Then, by using (i) where $\hat{A}$ and $\hat{\bold{B}}$ are replaced with $\hat{A}^T$ and $-\hat{\bold{B}}$, respectively,  there exists $\hat{w} \in H^{1,2}_0(U)_b$ such that
\begin{equation} \label{dualweigheq}
\int_U \langle \hat{A}^T \nabla \hat{w}, \nabla \varphi \rangle - \langle \hat{\bold{B}}, \nabla \hat{w} \rangle \varphi  +  \hat{c} \hat{w} \varphi  dx  = \int_{U} \phi \varphi dx, \quad \text{ for all $\varphi \in C_0^{\infty}(U)$}.
\end{equation}
Let $\hat{v}_n:= (\hat{v} \wedge n)\vee (-n)$. Then, we have $\hat{v}_n \in H^{1,2}_0(U)_b$ and $\lim_{n \rightarrow \infty} \hat{v}_n = \hat{v}$ in $H_0^{1,2}(U)$ by \cite[Proposition 4.17(i)]{MR92}
and $\lim_{n \rightarrow \infty} \hat{v}_n = \hat{v}$ a.e. on $U$ with $|\hat{v}_n| \leq |v|$ in $U$ for all $n \geq 1$.
Using an approximation, \lq\lq for all $\varphi \in C_0^{\infty}(U)$\rq\rq \; in \eqref{dualweigheq} is replaced by \lq\lq for all $\varphi \in H^{1,2}_0(U)_b$\rq\rq. Thus, replacing $\varphi \in C_0^{\infty}(U)$ in \eqref{dualweigheq} by $\hat{v}_n \in H^{1,2}_0(U)_b$, we get
\begin{equation} \label{dualweigheq2}
\int_U \langle \hat{A}^T \nabla \hat{w}, \nabla \hat{v}_n \rangle  - \langle \hat{\bold{B}}, \nabla \hat{w} \rangle \hat{v}_n   +  \hat{c}\hat{w} \hat{v}_n dx  = \int_{U} \phi \hat{v}_n  dx.
\end{equation}
Applying integration by parts to \eqref{dualweigheq2} and using \eqref{weakdivfree}, we have 
\begin{equation} \label{dualweigheq3}
\int_U \langle \hat{A} \nabla \hat{v}_n, \nabla \hat{w} \rangle  + \langle \hat{\bold{B}}, \nabla \hat{v}_n \rangle \hat{w}  +  \hat{c} \hat{v}_n \hat{w}  dx  = \int_{U} \phi \hat{v}_n  dx.
\end{equation}
Letting $n \rightarrow \infty$ in \eqref{dualweigheq3} and using the Lebesgue dominated convergence theorem, we get
\begin{equation} \label{dualweigheq4}
\int_U \langle \hat{A} \nabla \hat{v}, \nabla \hat{w} \rangle + \langle \hat{\bold{B}}, \nabla \hat{v} \rangle \hat{w}   + \hat{c} \hat{v}\hat{w} dx  = \int_{U} \phi \hat{v}  dx.
\end{equation}
Likewise, using an approximation we can replace $\varphi \in C_0^{\infty}(U)$ in \eqref{zeroforcdu} by $\hat{w} \in H^{1,2}_0(U)_b$, and hence we get
\begin{equation} \label{zeroforc2}
\int_U \langle \hat{A} \nabla \hat{v}, \nabla \hat{w} \rangle   +  \langle \hat{\bold{B}}, \nabla \hat{v} \rangle \hat{w}  + \hat{c}\hat{v} \hat{w} dx  = 0.
\end{equation}
Thus, \eqref{dualweigheq4} and \eqref{zeroforc2} imply 
$$
\int_{U} \hat{v} \phi dx = 0.
$$
Since $\phi \in C_0^{\infty}(U)$ is arbitrarily chosen, we get $\hat{v}=0$ in $U$, as desired.
\\ \\
(iii) Since $\hat{c} \geq \alpha$, we have $\hat{c}_n=\hat{c} \wedge n \geq \alpha$ for all $n \geq \alpha$. Thus, by applying
\cite[Theorem 1.1(ii)]{L24} to \eqref{ourweigheq}, we discover that for each $n \geq \alpha$ 
\begin{equation} \label{contraappro}
\|\hat{u}_n\|_{L^s(U)} \leq \frac{1}{\alpha} \|\hat{f}\|_{L^s(U)} +|U|^{1/2} \hat{K}_4 \| \bold{F} \|_{L^{2q}(U)}.
\end{equation}
Since $\lim_{n \rightarrow \infty}\hat{u}_n =\hat{u}$ a.e. (see \eqref{weakptconv}) and \eqref{bddesti} holds, letting $n \rightarrow \infty$ in \eqref{contraappro} the assertion follows by the Lebesgue dominated convergence theorem.
\end{proof}

\begin{rem} \label{imptremau}
Assume that all conditions of Theorem \ref{theomaineun} and let $\Lambda>0$ be a constant with $|U| \leq \Lambda$.
If one checks carefully the proofs of Theorem \ref{theomaineun} and \cite[Theorem 1.1(i)]{L24}, then $|U|$ involved in the dependency of the constants
$\hat{K}_3$ and $\hat{K}_4$ can be replaced by $\Lambda$.
\end{rem}

\centerline{}
\noindent
{\bf Proof of Theorem \ref{mainexisthm})}
(i) Let $\rho \in H^{1,2}(U) \cap C(\overline{U})$ be a strictly positive function on $\overline{U}$ with $\rho(x_1)=1$ for some $x_1 \in U$ as constructed in Theorem \ref{existrho} and let $\bold{B}:=\bold{H}+\frac{1}{\rho} A^T \nabla \rho$ in $U$. Then, $\rho \bold{B} \in L^2(U, \mathbb{R}^d)$ and
$$
\int_{U} \langle \rho \bold{B}, \nabla \varphi \rangle dx = 0, \quad \text{ for all $\varphi \in C_0^{\infty}(U)$}.
$$
By Theorem \ref{theomaineun}(i) there exists $u \in H^{1,2}_0(U)_b$ such that
\begin{equation} \label{chandiveq}
\int_U \langle \rho A \nabla u, \nabla \varphi \rangle  + \langle \rho \bold{B}, \nabla u \rangle \varphi  +  \rho cu \varphi dx  = \int_{U} f \rho \varphi  dx, \quad \text{ for all $\varphi \in C_0^{\infty}(U)$}
\end{equation}
and that
\begin{equation*}  \label{energes2}
\|u\|_{H_0^{1,2}(U)}  \leq K_3 \| f \rho \|_{L^{2_*}(U)} \leq K_3  \max_{\overline{U}} \rho  \, \| f \|_{L^{2_*}(U)} \leq  K_3  K_1   \| f \|_{L^{2_*}(U)},
\end{equation*}
\begin{equation*} \label{bddesti2}
\|u\|_{L^{\infty}(U)} \leq K_4 \|f \rho \|_{L^q(U)} \leq K_4 \max_{\overline{U}} \rho \|f\|_{L^q(U)}  \leq K_4  K_1 \|f\|_{L^q(U)}.
\end{equation*}
Indeed, $u$ is a weak solution to \eqref{maineq2} by Theorem \ref{mainequivthm}. \\
(ii)
Let $v \in H^{1,2}_0(U)$ with $cv \in L^1(U)$ satisfy \eqref{weakdefzerve} where $u$ is replaced by $v$. Then, applying Theorem \ref{mainequivthm}, \eqref{chandiveq} holds where $u$ is replaced by $v$. Thus, $u=v$ in $U$ by Theorem \ref{theomaineun}(ii). \\
(iii) Note that $\rho c \geq \alpha \min_{\overline{U}} \rho >0$. Then, applying Theorem \ref{theomaineun} to \eqref{chandiveq} and using Theorem \ref{existrho}(ii), we have
$$
\| u \|_{L^\theta(U)} \leq  \frac{1}{\alpha \min_{\overline{U}} \rho} \|\rho f\|_{L^{\theta}(U)} \leq \frac{1}{\alpha}\frac{\max_{\overline{U}}\rho}{\min_{\overline{U}} \rho} \,\|f\|_{L^\theta(U)} \leq \frac{K_1}{\alpha} \|f\|_{L^\theta(U)}.
$$
{\vskip 0.1cm \hfill$\Box$}

\begin{rem} \label{imptremaforme}
\begin{itemize}
\item[(i)]
Assume that all conditions of Theorem \ref{mainexisthm} hold and let $\Lambda>0$ be a constant with $|U| \leq \Lambda$. Then, as explained in Remark \ref{imptremau}, $|U|$ involved in the dependency of the constants $K_3$, $K_4$ can be replaced by $\Lambda$.
\item[(ii)]
In the proof of Theorem \ref{mainexisthm}, we can replace $\rho$ constructed in Theorem \ref{existrho} by a function $\tilde{\rho} \in H^{1,2}(U) \cap C(\overline{U})$ satisfying that $\rho(x)>0$ for all $x\in \overline{U}$ and that
\begin{equation} \label{infinitesiminv}
\int_{U} \langle A^T \nabla \tilde{\rho} + \tilde{\rho} \bold{H}, \nabla \psi \rangle dx = 0, \quad \text{ for all $\psi \in C_0^{\infty}(U)$}.
\end{equation}
Indeed, we can choose an open ball $B$ with $\overline{U} \subset B$. Then, it is straightforward to choose an arbitrary function $\tilde{\rho} \in H^{1,2}(B) \cap C(\overline{B})$  (for instance, we can simply choose $\tilde{\rho} \in C^1(\overline{B})$ ). Consequently, we have $\tilde{\rho} \in H^{1,2}(U) \cap C(\overline{U})$. Now choose an arbitrary vector field $\bold{E} \in L^p(B, \mathbb{R}^d)$ satisfying the weak divergence-free condition on $B$, i.e.
\begin{equation} \label{weakdiv}
\int_{B} \langle \bold{E}, \nabla \psi \rangle \, dx = 0, \quad \text{for all $\psi \in C_0^{\infty}(B)$}.
\end{equation}
Finally, define 
$$
\bold{H} := \frac{1}{\rho} \Big(\bold{E} - A^T \nabla \tilde{\rho} \Big).
$$
Then, it directly follows from \eqref{weakdiv} that
\begin{equation*} 
\int_{B} \langle A^T \nabla \tilde{\rho} + \tilde{\rho} \bold{H}, \nabla \psi \rangle \, dx = 0, \quad \text{for all $\psi \in C_0^{\infty}(B)$},
\end{equation*}
and hence \eqref{infinitesiminv} holds.
In that case, analogously to the proof of Theorem \ref{mainexisthm} where $\bold{B}$ is replaced by $\tilde{\bold{B}}:=\bold{H}+\frac{1}{\tilde{\rho}} A^T \nabla \tilde{\rho}$, we also get a unique solution $u \in H^{1,2}_0(U)_b$ to \eqref{maineq2} since it holds by the proof of Theorem \ref{theomaineun}(i) that
\begin{equation}  \label{divtypeequat}
\int_U \langle \tilde{\rho} A \nabla u, \nabla \varphi \rangle  + \langle \tilde{\rho} \tilde{\bold{B}}, \nabla u \rangle \varphi  +  \tilde{\rho} cu \varphi dx  = \int_{U} f \tilde{\rho} \varphi  dx, \quad \text{ for all $\varphi \in C_0^{\infty}(U)$}.
\end{equation}
In particular, if $c \geq \alpha$ in $U$ for some constant $\alpha \in (0, \infty)$, then applying Theorem \ref{theomaineun} to \eqref{divtypeequat} we obtain that for each $f \in L^q(U) \cap L^{\theta}(U)$ with $q \in (\frac{d}{2}, \infty)$ and $\theta \in [1, \infty]$,
\begin{equation} \label{imptcontrac}
\| u \|_{L^\theta(U)} \leq  \frac{1}{\alpha \min_{\overline{U}} \tilde{\rho}} \| \tilde{\rho} f\|_{L^{\theta}(U)} \leq \frac{1}{\alpha}\frac{\max_{\overline{U}}\tilde{\rho}}{\min_{\overline{U}} \tilde{\rho}} \,\|f\|_{L^\theta(U)}.
\end{equation}
Thus, $K_1>0$ in Theorem \ref{mainexisthm}(iii) can be replaced by any constant $\tilde{K}_1>0$ satisfying $\frac{\max_{\overline{U}}\tilde{\rho}}{\min_{\overline{U}} \tilde{\rho}} \leq \tilde{K}_1$.
\item[(iii)]
One of the main ingredients for deriving Theorem \ref{mainexisthm} is Theorem \ref{existrho}, where the key is the existence of a strictly positive and continuous function $\rho$ satisfying \eqref{infinva}. For critical drifts such as $\bold{H} \in L^d(U, \mathbb{R}^d)$, constructing such a regular $\rho$ may not be possible, making it difficult to establish the results in Theorem \ref{mainexisthm}. Identifying drifts under more general conditions than $\bold{H} \in L^p(U, \mathbb{R}^d)$, where results in Theorem \ref{mainexisthm} holds, remains a challenging problem.
\end{itemize}
\end{rem}

\subsection{The case of $d=2$, $c \in L^1(U)$, $f \in L^{q}(U)$, $\bold{F} \in L^{2q}(U, \mathbb{R}^d)$ with $q \in (1, \infty)$}
\begin{theo} \label{mainexisthm2n}
Assume that {\bf (S)} holds with $d=2$, $c \in L^1(U)$ with $c \geq 0$ in $U$,  $f \in L^q(U)$ and $\bold{F} \in L^{2q}(U, \mathbb{R}^d)$ with $q \in (1, \infty)$. Then, the following hold:
\begin{itemize}
\item[(i)]
There exists $u \in H^{1,2}_0(U)_b$ such that $u$ is a unique weak solution to 
\begin{equation} \label{maineq22n}
\left\{
\begin{alignedat}{2}
-{\rm div}(A \nabla u) + \langle \bold{H}, \nabla u \rangle + c u &=f -{\rm div} \bold{F} && \quad \; \mbox{\rm in $U$}\\
u &= 0 &&\quad \; \mbox{\rm on $\partial U$},
\end{alignedat} \right.
\end{equation}
 i.e. $u \in H^{1,2}_0(U)$ with $cu \in L^1(U)$ satisfies
\begin{equation} \label{weakdefzerve2n}
\int_U \langle A \nabla u, \nabla \psi \rangle + \langle \bold{H}, \nabla u \rangle \psi   + cu \psi dx = \int_U f \psi dx + \int_{U} \langle \bold{F}, \nabla \psi \rangle dx \quad \text{ for all } \psi \in C_0^{\infty}(U)
\end{equation}
and if $v \in H^{1,2}_0(U)$ with $cv \in L^1(U)$ satisfies \eqref{weakdefzerve2n} where $v$ is replaced by $u$, then $u=v$ in $U$.

\item[(ii)] $u$ in Theorem \ref{mainexisthm2n}(i) satisfies that
\begin{equation*} 
\|u\|_{H_0^{1,2}(U)}  \leq K_7 \left( \|f\|_{L^{q}(U)} + \| \bold{F} \|_{L^{2q}(U)} \right)
\end{equation*}
and
\begin{equation*}
\|u\|_{L^{\infty}(U)} \leq K_8 \left(\|f\|_{L^q(U)} + \|\bold{F} \|_{L^{2q}(U)} \right),
\end{equation*}
where $K_7:= K_1 \Big( K_3 |U|^{\frac{1}{2}-\frac{1}{2q}}  \vee (K_2K_3+|U|^{\frac{1}{2}-\frac{1}{2q}}) \Big)$, $K_8:= K_1 \Big(  K_4  |U|^{\frac12 -\frac{1}{2q}}  \vee K_2 K_4 \vee   |U|^{\frac14 -\frac{1}{4q}}  \Big)$ and $K_1, K_2>0$ are constants as in Theorem \ref{existrho} and $K_3, K_4>0$ are constants as in Theorem \ref{mainexisthm}.

\end{itemize}
\end{theo}
\begin{proof}
(i)
Let $\rho \in H^{1,2}(U) \cap C(\overline{U})$ be constructed in Theorem \ref{existrho} and let $\bold{B}:=\bold{H}+\frac{1}{\rho} A^T \nabla \rho$. Note that 
$$
\hat{q}:=\left(\frac{1}{2} +\frac{1}{2q} \right)^{-1} \in (1,q \wedge 2), \; \text{ and } \; \langle \nabla \rho, \bold{F} \rangle \in L^{\hat{q}}(U).
$$
Then, $\rho f  + \langle \nabla \rho, \bold{F} \rangle \in L^{\hat{q}}(U)$ and $\rho \bold{F} \in L^{2q}(U, \mathbb{R}^d) 
\subset L^{2\hat{q}}(U, \mathbb{R}^d)$. By Theorem \ref{theomaineun} where $q$ is replaced by $\hat{q}$, there exists a unique weak solution $u \in H^{1,2}_0(U)_b$ to
\begin{equation} \label{mainenerbd2n}
\left\{
\begin{alignedat}{2}
-{\rm div}(\rho A \nabla u) + \langle \rho \bold{B}, \nabla u \rangle +\rho c u &= \Big( \rho f + \langle \nabla \rho, \bold{F} \rangle \Big) - {\rm div} (\rho \bold{F}) && \quad \mbox{in $U$}\\
u &= 0 &&\quad \mbox{on $\partial U$},
\end{alignedat} \right.
\end{equation}
i.e. $u$ satisfies
\begin{align}
&\int_U \langle \rho A \nabla u, \nabla \varphi \rangle  + \langle \rho \bold{B}, \nabla u \rangle \varphi  +  \rho cu \varphi dx  = \int_{U} \Big( f \rho  + \langle \nabla \rho, \bold{F} \rangle \Big) \varphi dx + \int_{U} \langle \rho\bold{F}, \nabla \varphi \rangle dx \nonumber \\
& \quad  = \int_{U} \rho f \varphi dx  + \int_{U} \langle \bold{F}, \nabla (\rho \varphi) \rangle dx, \quad \text{ for all $\varphi \in C_0^{\infty}(U)$} 	\label{chandiveq2n}
\end{align}
and if $v \in H^{1,2}(U)$ with $\rho cv \in L^1(U)$ satisfies \eqref{chandiveq2n} where $u$ is replaced by $v$, then $u=v$ in $U$. Applying Theorem \ref{mainequivthm} to \eqref{mainenerbd2n} and using Theorem \ref{theomaineun}, we discover that $u \in H^{1,2}_0(U)_b$ is a unique solution to \eqref{maineq22n}. \\
(ii) Applying \eqref{estimenbd} in Theorem \ref{theomaineun} to \eqref{mainenerbd2n} and using Theorem \ref{existrho}, we get
\begin{align*}
\| u\|_{H^{1,2}_0(U)} &\leq K_3 \left( \|\rho f + \langle \nabla \rho, \bold{F} \rangle \|_{L^{\hat{q}}(U)}  + \| \rho \bold{F} \|_{L^2(U)} \right) \\
&\leq K_3 K_1 \|f\|_{L^{\hat{q}}(U)} + K_1 K_2 K_3  \|\bold{F} \|_{L^{2q}(U)}  +K_1 \| \bold{F} \|_{L^2(U)} \\
&\leq K_7 \left( \|f\|_{L^{q}(U)}	 + \| \bold{F} \|_{L^{2q}(U)} \right).
\end{align*}
Likewise,  applying \eqref{bddestimdivf} in Theorem \ref{theomaineun} to \eqref{mainenerbd2n} and using Theorem \ref{existrho}, we have
\begin{align*}
\| u\|_{L^{\infty}(U)} &\leq K_4 \left( \|\rho f + \langle \nabla \rho, \bold{F} \rangle \|_{L^{\hat{q}}(U)}  + \| \rho \bold{F} \|_{L^{2\hat{q}}(U)} \right) \\
&\leq K_4 K_1 \|f\|_{L^{\hat{q}}(U)} + K_1 K_2 K_4  \| \bold{F} \|_{L^{2q}(U)}   + K_1 \| \bold{F}\|_{L^{2\hat{q}}(U)} \\
& \leq  K_1 K_4  |U|^{\frac12 -\frac{1}{2q}} \| f\|_{L^q(U)} +K_1 K_2 K_4  \| \bold{F} \|_{L^{2q}(U)}  +K_1 |U|^{\frac{1}{4}-\frac{1}{4q}}  \| \bold{F} \|_{L^{2q}(U)} \\
&\leq K_8 \left( \|f\|_{L^{q}(U)}	 + \| \bold{F} \|_{L^{2q}(U)} \right).
\end{align*}
\end{proof}

\subsection{The case of $d=3$, $c \in L^1(U)$, $f \in L^{q}(U)$, $\bold{F} \in L^{\theta}(U, \mathbb{R}^d)$ with $\theta \in (6, \infty)$, $q \in (\frac{3}{2}, \frac{2\theta}{2+\theta}]$}
\begin{theo} \label{specialresd}
Assume that {\bf (S)} holds with $d=3$, $c \in L^1(U)$ with $c \geq 0$ in $U$, $f \in L^q(U)$, and $\bold{F} \in L^{\theta}(U, \mathbb{R}^d)$  with $\theta \in (6, \infty)$ and $q \in (\frac{d}{2}, \frac{2\theta}{2+\theta}]$.
Then, the following hold: 
\begin{itemize}
\item[(i)]
There exists $u \in H^{1,2}_0(U)_b$ such that $u$ is a unique weak solution to  \eqref{maineq22n}.

\item[(ii)] $u$ in Theorem \ref{specialresd}(i) satisfies that
\begin{equation*} 
\|u\|_{H_0^{1,2}(U)}  \leq K_{10} \left( \|f\|_{L^{2_*}(U)} + \| \bold{F} \|_{L^{d}(U)} \right)
\end{equation*}
and
\begin{equation*}
\|u\|_{L^{\infty}(U)} \leq K_{11} \left(\|f\|_{L^q(U)} + \|\bold{F} \|_{L^{\theta}(U)} \right),
\end{equation*}
where $K_{10}:= K_1 \Big( K_3 \vee (K_2K_3+|U|^{\frac{1}{2}-\frac{1}{d}}) \Big)$, \;$K_{11}:= K_1 \Big(  K_4   \vee  |U|^{\frac{1}{q}-\frac{2+\theta}{2\theta}} K_2 K_4 \vee  |U|^{\frac{1}{2q} -\frac{1}{\theta}}  \Big)$, $K_1, K_2>0$ are constants as in Theorem \ref{existrho} and $K_3, K_4>0$ are constants as in Theorem \ref{mainexisthm}.
\end{itemize}
\end{theo}

\noindent
\begin{proof}
(i)
Since $2q \leq \frac{4 \theta}{2+\theta}<\theta$, we get $\rho \bold{F} \in L^{\theta}(U, \mathbb{R}^d) \subset L^{2q}(U, \mathbb{R}^d)$
and $\rho f  + \langle \nabla \rho, \bold{F} \rangle \in L^{q}(U)$.
As in the proof of Theorem \ref{mainexisthm2n}, it follows from Theorem \ref{theomaineun} that there exists $u \in H^{1,2}_0(U)_b$ such that $u$ is a unique weak solution to \eqref{mainenerbd2n}, and hence $u$ is a unique weak solution to \eqref{maineq22n} by Theorem \ref{mainequivthm}. \\
(ii) Applying \eqref{estimenbd} in Theorem \ref{theomaineun} to \eqref{mainenerbd2n} and using Theorem \ref{existrho}, we then obtain that 
\begin{align*}
\| u\|_{H^{1,2}_0(U)} &\leq K_3 \left( \|\rho f + \langle \nabla \rho, \bold{F} \rangle \|_{L^{2_*}(U)}  + \| \rho \bold{F} \|_{L^2(U)} \right) \\
&\leq K_3 K_1 \|f\|_{L^{2_*}(U)} + K_1 K_2 K_3  \|\bold{F} \|_{L^{d}(U)}  +K_1 \| \bold{F} \|_{L^2(U)} \\
&\leq K_{10} \left( \|f\|_{L^{2_*}(U)}	 + \| \bold{F} \|_{L^{d}(U)} \right)
\end{align*}
and that
\begin{align*}
\| u\|_{L^{\infty}(U)} &\leq K_4 \left( \|\rho f + \langle \nabla \rho, \bold{F} \rangle \|_{L^{q}(U)}  + \| \rho \bold{F} \|_{L^{2q}(U)} \right) \\
&\leq K_4 K_1 \|f\|_{L^{q}(U)} + |U|^{\frac{1}{q}-\frac{2+\theta}{2\theta}} K_4  \|\langle \nabla \rho, \bold{F} \rangle \|_{L^{\frac{2\theta}{2+\theta}}(U)}  + K_1 \| \bold{F}\|_{L^{2q}(U)} \\
& \leq  K_1 K_4  \| f\|_{L^q(U)} +  |U|^{\frac{1}{q}-\frac{2+\theta}{2\theta}}K_1 K_2 K_4  \| \bold{F} \|_{L^{\theta}(U)}  +K_1 |U|^{\frac{1}{2 q}-\frac{1}{\theta}}  \| \bold{F} \|_{L^{\theta}(U)} \\
&\leq K_{11} \left( \|f\|_{L^{q}(U)}	 + \| \bold{F} \|_{L^{2q}(U)} \right),
\end{align*}
as desired.

\end{proof}

\section{Divergence type equations with $VMO$ coefficients of $A$} \label{vmow1pdivse}
\begin{theo} \label{improinfinis}
Assume that {\bf (S)} holds. Let $\rho$ be the same as the one in Theorem \ref{existrho}. Assume that $a_{ij} \in VMO_{\omega}$ for all $1 \leq i,j \leq d$. Then, $\rho \in H^{1,p}(U)$ and there exists a constant $\tilde{C}>0$ which only depends on $d$, $\lambda$, $M$, $r$, $p$, $\|h\|_{L^p(U)}$ and $\omega$ such that
$$
\| \nabla \rho \|_{L^p(U)} \leq \| \nabla \rho \|_{L^p(B_r(x_0))}  \leq  \tilde{C}.
$$
\end{theo}
\begin{proof}
Let $\chi \in  H^{1,\infty}(\mathbb{R}^d)_0 \cap H^{1,2}_0(B_{2r}(x_0))$ with $\text{supp}(\chi) \subset B_{2r}(x_0)$ 
and $\phi \in C_0^{\infty}(B_{2r}(x_0))$ be arbitrary. Then,
\begin{align*}
&\int_{B_{2r}(x_0)} \langle A^T \nabla \rho, \nabla (\chi \phi) \rangle \,dx = \int_{B_{2r}(x_0)} \langle A^T \nabla \rho, \nabla \phi \rangle \chi \,dx  + \int_{B_{2r}(x_0)} \langle A^T \nabla \rho, \nabla \chi \rangle \phi\,dx \\
& = \int_{B_{2r}(x_0)} \langle A^T (\chi \rho), \nabla \phi \rangle\,dx - \int_{B_{2r}(x_0)} \langle \rho A^T \nabla \chi, \nabla \phi \rangle  \,dx+  \int_{B_{2r}(x_0)} \langle A^T \nabla \rho, \nabla \chi \rangle \phi \,dx
\end{align*}
and
\begin{align*}
\int_{B_{2r}(x_0)} \langle \rho \bold{H}, \nabla (\chi \phi ) \rangle dx =  \int_{B_{2r}(x_0)} \langle (\chi \rho) \bold{H}, \nabla \phi  \rangle dx  + \int_{B_{2r}(x_0)} \langle \rho \bold{H}, \nabla \chi  \rangle \phi\, dx.
\end{align*}
Thus, replacing $\varphi$ by $\phi \chi$ in \eqref{infinva}, we derive that
\begin{align}
&\int_{B_{2r}(x_0)} \langle A^T (\chi \rho), \nabla \phi \rangle  dx =  \int_{B_{2r}(x_0)} \langle \rho A^T \nabla \chi -\chi \rho \bold{H} , \nabla \phi \rangle  dx \nonumber \\
& \qquad \qquad +  \int_{B_{2r}(x_0)} \langle -\rho \bold{H}-A^T \nabla \rho, \nabla \chi \rangle \phi dx, \quad \forall \phi \in C_0^{\infty}(B_{2r}(x_0)). \label{crucidenti}
\end{align}
Let $\hat{\rho}:=\chi \rho$,\; $\tilde{\bold{F}}:= \rho A^T \nabla \chi -\chi \rho \bold{H}$ and $\hat{f}:=\langle -\rho \bold{H}-A^T \nabla \rho, \nabla \chi \rangle$ on $B_{2r}(x_0)$. Extend $\hat{\rho}$,  $\tilde{\bold{F}}$ and $\hat{f}$ on $\mathbb{R}^d$ by the zero-extension, and say again $\hat{\rho}$, $\tilde{\bold{F}}$ and $\hat{f}$, respectively. Since $\hat{\rho}$, $\tilde{\bold{F}}$ and $\hat{f}$ have compact supports in $B_{2r}(x_0)$, we then obtain that for any $\hat{\lambda} \geq 0$
\begin{equation} \label{underlyingequ}
\int_{\mathbb{R}^d} \langle A^T \nabla \hat{\rho}, \nabla \phi \rangle \,dx  + \hat{\lambda} \int_{\mathbb{R}^d} \hat{\rho} \phi \,dx = \int_{\mathbb{R}^d} \langle \tilde{\bold{F}}, \nabla \phi \rangle \,dx + \int_{\mathbb{R}^d} \Big(\hat{f} + \hat{\lambda} \hat{\rho} \Big) \phi \,dx, \quad \forall \phi \in C_0^{\infty}(\mathbb{R}^d).
\end{equation}
{\sf \underline{Case 1})} Assume that $\frac{1}{2} \leq \frac{1}{p}+\frac{1}{d}$. Choosing $\chi$ with $\chi \equiv 1$ on $\overline{B}_r(x_0)$ and $0 \leq  \chi(x) \leq 1$ for all $x \in \mathbb{R}^d$ and applying Theorem \ref{krylovmainthm} to \eqref{underlyingequ}, we establish that
$\hat{\rho} \in H^{1,p}(\mathbb{R}^d)$ and there exists $\lambda_0>0$ which only depends on $d$, $\lambda$, $M$, $p$ and $\omega$ such that
\begin{align*}
&\| \chi \rho \|_{H^{1,p}(B_{2r}(x_0))}= \|\hat{\rho} \|_{H^{1,p}(\mathbb{R}^d)}\\
 &\leq \tilde{N}_0 \left( \| \tilde{\bold{F}}\|_{L^p(\mathbb{R}^d)} + \| \hat{f} +\lambda_0 \hat{\rho}\|_{L^{\frac{pd}{p+d}}(\mathbb{R}^d)}	\right) \\
& \leq  \tilde{N}_0 \Bigg(	  K_1\Big(dM \|\nabla \chi \|_{L^{\infty}(\mathbb{R}^d)} + \|h\|_{L^p(U)} \Big) + |B_{2r}(x_0)|^{\frac{1}{p}+\frac{1}{d}-\frac{1}{2}} \| \nabla \chi \|_{L^{\infty}(\mathbb{R}^d)} \Big( K_1 |U|^{\frac{1}{2}-\frac{1}{p}}\|h \|_{L^p(U)} +	dM K_1 K_2\Big)  \\
& \qquad \qquad + \lambda_0 |B_{2r}(x_0)|^{\frac{1}{p}+\frac{1}{d}} K_1  \Bigg),
\end{align*}
 $\tilde{N}_0>0$ is a constant which only depends on $d$, $\lambda$, $M$, $p$ and $\omega$. \\ \\
{\sf \underline{Case 2})} Assume that $\frac{1}{p}+\frac{1}{d} < \frac12$. Then, $d \geq 3$. Choose a unique $k \in \mathbb{N} \cup \{0\}$ such that
$$
\frac{1}{p}+\frac{1}{d} \in \left[	\frac{1}{2} - \frac{k+1}{d}, \frac{1}{2} - \frac{k}{d}	\right).
$$
For each $m:=0, 1, \ldots, k+2$, let $r_m:= 2r-\frac{rm}{k+2}$. For each $m=1,2, \ldots, k+2$, let $\chi_m \in  H^{1,\infty}(\mathbb{R}^d)_0 \cap H^{1,2}_0(B_{r_{m-1}}(x_0))$ with $\text{supp}(\chi_m) \subset B_{r_{m-1}}(x_0)$ and $\chi_m \equiv 1$ on $\overline{B}_{r_m}(x_0)$. Replacing $B_{2r}(x_0)$ and $\chi$ by $B_{r_{m-1}}(x_0)$ and $\chi_m$ in \eqref{crucidenti}, respectively, and using an iterative argument with Theorem \ref{krylovmainthm}, for each $m=1,2, \ldots k+1$
we get $\chi_m \rho \in H^{1,\left( \frac{1}{2}-\frac{m}{d}\right)^{-1}}(B_{r_{m-1}}(x_0))$ and in particular, we have
\begin{equation} \label{alfinestiwp}
\|\chi_m \rho \|_{H^{1,\left( \frac{1}{2}-\frac{m}{d}\right)^{-1}}(B_{r_{m-1}}(x_0))} \leq \tilde{C}_m,
\end{equation}
where $\tilde{C}_m>0$ is a constant which only depends on $d$, $\lambda$, $M$, $r$, $p$, $\|h\|_{L^p(U)}$ and $\omega$.  Thus, it directly follows that $\chi_{k+1}\rho \in H^{1,\left( \frac{1}{p}+\frac{1}{d} \right)^{-1}}(B_{r_{k}}(x_0))$. Applying Theorem \ref{krylovmainthm} to \eqref{crucidenti} where $B_{2r}(x_0)$ and  $\chi$ are replaced by $B_{r_{k+1}}(x_0)$ and $\chi_{k+2}$, respectively, we finally obtain from \eqref{alfinestiwp} that $\chi_{k+2} \rho \in H^{1,p}(B_{r_{k+1}}(x_0))$ and
$$
\|\rho \|_{H^{1,p}(B_r(x_0))} \leq \| \chi_{k+2} \rho \|_{H^{1,p}(B_{r_{k+1}}(x_0))} \leq \tilde{C}_{k+2},
$$
where $\tilde{C}_{k+2}>0$ is a constant which only depends on $d$, $\lambda$, $M$, $r$, $p$, $\|h\|_{L^p(U)}$ and $\omega$.
\end{proof}

\begin{theo} \label{vmouniqweaksol}
Assume that {\bf (S)} holds, $c \in L^1(U)$ with $c \geq 0$ in $U$, $a_{ij} \in VMO_{\omega}$ for all $1 \leq i,j \leq d$, $f \in L^q(U)$ and $\bold{F} \in L^{\gamma}(U, \mathbb{R}^d)$  with $q \in (\frac{d}{2}, \infty)$ and $\gamma \in (d, \infty)$.
Then, the following hold: 
\begin{itemize}
\item[(i)]
There exists a unique weak solution $u \in H^{1,2}_0(U)_b$  to \eqref{maineq22n}. 
\item[(ii)] We have
\begin{equation*} 
\|u\|_{H_0^{1,2}(U)}  \leq K_{12} \left( \|f\|_{L^{q}(U)} + \| \bold{F} \|_{L^{\gamma}(U)} \right)
\end{equation*}
and
\begin{equation*}
\|u\|_{L^{\infty}(U)} \leq K_{13} \left(\|f\|_{L^q(U)} + \|\bold{F} \|_{L^{\gamma}(U)} \right),
\end{equation*}
\end{itemize}
where $2_*$ is an arbitrarily fixed number in $(1, \frac{p\gamma}{p+\gamma} \wedge 2)$ if $d=2$, $K_{12} := K_1 K_3 \vee \Big(|U|^{\frac{1}{2_*} - \frac{1}{p}-\frac{1}{\gamma}} \tilde{C} + K_1 |U|^{\frac12 -\frac{1}{\gamma}} \Big)$, $\hat{q} \in (\frac{d}{2}, q \wedge \frac{p\gamma}{p+\gamma} \wedge \frac{\gamma}{2})$, $K_{13}:= K_1 K_4 |U|^{\frac{1}{\hat{q}} - \frac{1}{q}} \vee\Big( |U|^{\frac{1}{\hat{q}} - \frac{1}{\gamma} - \frac{1}{p}}  K_4 \tilde{C}  + |U|^{\frac{1}{2\hat{q}} - \frac{1}{\gamma}} K_1 \Big)$, $K_1 \geq 1$ is a constant as in Theorem \ref{existrho}, $K_3, K_4>0$ are constants as in Theorem \ref{mainexisthm} and $\tilde{C}>0$ is a constant as in Theorem \ref{improinfinis}.

\end{theo}
\begin{proof}
(i) Let $\rho$ be as in Theorem \ref{existrho}. Then, by Theorem \ref{improinfinis}, $\rho \in H^{1,p}(U)$ and
$$
\| \nabla \rho \|_{L^p(U)} \leq \tilde{C},
$$
where $\tilde{C}>0$ is a constant which only depends on $d$, $\lambda$, $M$, $r$, $p$, $\|h\|_{L^p(U)}$ and $\omega$.
Choose $\hat{q} \in (\frac{d}{2}, q \wedge \frac{p\gamma}{p+\gamma} \wedge \frac{\gamma}{2})$. Then, $\rho f  + \langle \nabla \rho, \bold{F} \rangle \in L^{q \wedge \frac{p\gamma}{p+\gamma}}(U) \subset L^{\hat{q}}(U)$ and $\rho \bold{F} \in L^{\gamma}(U, \mathbb{R}^d) \subset L^{2\hat{q}}(U, \mathbb{R}^d)$. By Theorem \ref{theomaineun}, there exists a unique weak solution $u \in H^{1,2}_0(U)_b$ to \eqref{mainenerbd2n} (cf. \eqref{chandiveq2n}). Thus, applying Theorem \ref{mainequivthm}, we discover that $u \in H^{1,2}_0(U)_b$ is a unique weak solution to \eqref{maineq22n}. \\ \\
(ii)
Applying \eqref{estimenbd} in Theorem \ref{theomaineun} to \eqref{mainenerbd2n} and using Theorem \ref{existrho}, we get
\begin{align*}
\| u\|_{H^{1,2}_0(U)} &\leq K_3 \left( \|\rho f + \langle \nabla \rho, \bold{F} \rangle \|_{L^{2_*}(U)}  + \| \rho \bold{F} \|_{L^2(U)} \right) \\
&\leq K_3 K_1 \|f\|_{L^{2_*}(U)} + |U|^{\frac{1}{2_*} - \frac{1}{p}-\frac{1}{\gamma}} \| \nabla \rho \|_{L^p(U)}\|\bold{F} \|_{L^{\gamma}(U)}  +K_1 |U|^{\frac{1}{2}-\frac{1}{\gamma}}  \| \bold{F} \|_{L^{\gamma}(U)} \\
&\leq K_{12} \left( \|f\|_{L^{2_*}(U)}	 + \| \bold{F} \|_{L^{\gamma}(U)} \right).
\end{align*}
 Likewise, applying \eqref{bddestimdivf} in Theorem \ref{theomaineun} to \eqref{mainenerbd2n} and using Theorem \ref{existrho}, we have
\begin{align*}
\| u\|_{L^{\infty}(U)} &\leq K_4 \left( \|\rho f + \langle \nabla \rho, \bold{F} \rangle \|_{L^{\hat{q}}(U)}  + \| \rho \bold{F} \|_{L^{2\hat{q}}(U)} \right) \\
&\leq K_4 K_1 |U|^{\frac{1}{\hat{q}} - \frac{1}{q}} \|f\|_{L^{q}(U)} +|U|^{\frac{1}{\hat{q}} - \frac{1}{\gamma} - \frac{1}{p}}  K_4   \|\nabla \rho \|_{L^p(U)}\| \bold{F} \|_{L^{\gamma}(U)}   + |U|^{\frac{1}{2\hat{q}} - \frac{1}{\gamma}} K_1 \| \bold{F}\|_{L^{\gamma}(U)} \\
&\leq K_{13} \left( \|f\|_{L^{q}(U)}	 + \| \bold{F} \|_{L^{\gamma}(U)} \right).
\end{align*}
\end{proof}



\begin{theo}
Assume that {\bf (S)} holds, $a_{ij} \in VMO_{\omega}$ for all $1 \leq i,j \leq d$, $c \in L^s(U)$ with $c \geq 0$ in $U$, $f \in L^q(U)$ and $\bold{F} \in L^{\gamma}(U, \mathbb{R}^d)$, where
$s \in (1, \infty)$, $q \in (\frac{d}{2}, \infty)$ and $\gamma \in (d, \infty)$. Let $\tilde{c} \in L^s(U)$ with $|c| \leq | \tilde{c}|$ in $U$
and $\beta:= s \wedge q \wedge \frac{d \gamma}{d+\gamma} \in (1, d)$. Let $u$ be a unique weak solution to \eqref{maineq22n} as in Theorem \ref{vmouniqweaksol}. 
Additionally, assume that $\partial U$ is of class $C^1$. Then, $u \in H^{1, \frac{d\beta}{d-\beta}}(U) \cap H^{1,2}_0(U)_b$ and
$$
\|u\|_{H^{1, \frac{d\beta}{d-\beta}}(U)} \leq Q_1 K_{15}\left( \|f\|_{L^q(U)} + \|\bold{F}\|_{L^{\gamma}(U)} 	\right),
$$
where $Q_1>0$ is a constant independent of $f$, $\bold{F}$, $c$, $\tilde{c}$ and $u$ and $K_{15}>0$ is a constant which only depends on $d, \lambda, M, r, p, \|h\|_{L^p(U)}$ and $\| \tilde{c}\|_{L^s(U)}$. 
\end{theo}
\begin{proof}
Observe that $u \in H^{1,2}_0(U)_b$ satisfies
\begin{equation} \label{varieineqvm}
\int_U \langle A \nabla u, \nabla \psi \rangle + \langle \bold{H}, \nabla u \rangle \psi dx = \int_U (f-cu) \psi +\langle \bold{F}, \nabla \psi \rangle dx, \quad \text{ for all } \psi \in C_0^{\infty}(U).
\end{equation}
Since $\beta \leq s \wedge q$ and $\beta \leq \frac{d\gamma}{d+\gamma}$, we get $f-cu \in L^{\beta}(U)$ and $\bold{F} \in L^{\frac{d\beta}{d-\beta}}(U, \mathbb{R}^d)$. By Corollary \ref{essenbdddiveq}, there exists $\hat{\bold{F}} \in L^{\frac{d\beta}{d-\beta}}(U, \mathbb{R}^d)$ such that
\begin{equation} \label{fcudivequ}
\int_{U} \langle \hat{\bold{F}}, \nabla \psi \rangle dx = \int_{U} (f-cu ) \psi dx, \quad \text{ for all $\psi \in C_0^{\infty}(U)$}
\end{equation}
and additionally Theorem \ref{vmouniqweaksol}(ii) deduces that
\begin{align*}
\|\hat{\bold{F}} \|_{L^{\frac{d\beta}{d-\beta}}(U)} &\leq \hat{C}(d, \beta) \|f-cu \|_{L^{\beta}(U)} \leq \hat{C}(d, \beta) \Big( |U|^{\frac{1}{\beta}-\frac{1}{q}} \|f\|_{L^q(U)} +\|\tilde{c}\|_{L^{\beta}(U)} \|u\|_{L^{\infty}(U)} \Big) \\
&\leq K_{14}\left( \|f\|_{L^q(U)} + \| \bold{F} \|_{L^{\gamma}(U)} \right),
\end{align*}
where $K_{14}:=\hat{C}(d, \beta)  \Big( |U|^{\frac{1}{\beta}-\frac{1}{q}} + K_{13} \|\tilde{c}\|_{L^{\beta}(U)} \Big)$, and $\hat{C}(d, \beta)>0$ and $K_{13}>0$ are constant as in Corollary \ref{essenbdddiveq} and Theorem \ref{vmouniqweaksol}(ii), respectively. Since $\frac{d\beta}{d-\beta}>\frac{d}{d-1}$, applying \cite[Theorem 2.1(i)]{KK19} to \eqref{varieineqvm} with \eqref{fcudivequ}, there exists $\tilde{u} \in H^{1,\frac{d\beta}{d-\beta}}_0(U)$ such that
\begin{equation*} 
\int_U \langle A \nabla \tilde{u}, \nabla \psi \rangle + \langle \bold{H}, \nabla \tilde{u} \rangle \psi dx = \int_U  \langle \hat{\bold{F}}+\bold{F}, \nabla \psi \rangle dx, \quad \text{ for all } \psi \in C_0^{\infty}(U).
\end{equation*}
and that
$$
\|\tilde{u}\|_{H^{1, \frac{d\beta}{d-\beta}}(U)} \leq Q_1  \| \hat{\bold{F}}+\bold{F} \|_{L^{\frac{d\beta}{d-\beta}}(U)} \leq Q_1  K_{15}\left( \|f	\|_{L^q(U)} + \| \bold{F} \|_{L^{\gamma}(U)}\right),
$$
where $Q_1>0$ is a constant independent of $f$, $\bold{F}$, $c$, $\tilde{c}$, $u$ and $K_{15}:=\left( K_{14} + |U|^{\frac{1}{\beta}- \frac{1}{d} - \frac{1}{\gamma}} \right) $. To complete our proof, it is enough to show $u=\tilde{u}$ in $U$. By the uniqueness result in \cite[Theorem 2.1(i)]{KK19}, $u=\tilde{u}$, which proves $u \in H^{1,\frac{d\beta}{d-\beta} \wedge 2}_0(U)$.
\end{proof}

\section{Non-divergence type equations with $VMO$ coefficients of $A$} \label{nondivresec}
\begin{theo} \label{vmonondivh1p}
Assume that {\bf (S)} holds, $c \in L^s(U)$ with $c \geq 0$ in $U$ and $f \in L^q(U)$, where $s \in (1, \infty)$ and $q \in (\frac{d}{2}, \infty)$.
Let $\beta:= p \wedge s \wedge q$ and $\tilde{c} \in L^s(U)$ with $|c| \leq \tilde{c}$ in $U$. Let $u \in H^{1,2}_0(U)_b$ be a unique weak solution to \eqref{maineq2} as in Theorem \ref{mainexisthm}. Additionally, assume that $\partial U$ is of class $C^{1,1}$.
Then, the following hold:
\begin{itemize}

\item[(i)] If $a_{ij} \in VMO_{\omega}$ for all $1 \leq i,j \leq d$ and ${\rm div} A \in L^p(U, \mathbb{R}^d)$ with $\|{\rm div}\,A \| \leq h$ in $U$, then $u \in H^{2, \beta}(U) \cap H^{1,2}_0(U)_b$ and
$$
\|u\|_{H^{2,\beta}(U) } \leq K_{16} \|f\|_{L^q(U)},
$$
where $K_{16}>0$ is a constant which only depends on $d$, $\lambda$, $M$, $r$, $p$, $s$, $q$, $\|h\|_{L^p(U)}$, $\|\tilde{c}\|_{L^s(U)}$ and $\omega$.

\item[(ii)]
If $a_{ij} \in VMO_{\omega}$ for all $1 \leq i,j \leq d$ and ${\rm div}A \in L^d(U, \mathbb{R}^d)$ with $d \geq 3$ and $\beta<d$, then $u \in H^{2, \beta}(U) \cap H^{1,2}_0(U)_b$  and the following estimate holds:
$$
\|u\|_{H^{2, \beta}(U)} \leq Q_2 K_{17} \|f\|_{L^q(U)},
$$
where $Q_2>0$ is a constant independent of $f$, $c$, $\tilde{c}$ and $u$, $K_{17}:=|U|^{\frac{1}{\beta}-\frac{1}{q}} +|U|^{\frac{1}{\beta}-\frac{1}{s}} \| \tilde{c}\|_{L^s(U)}  K_5$ and $K_5>0$ is a constant as in Theorem \ref{mainexisthm}.

\item[(iii)]
Let $\tilde{A} =(\tilde{a}_{ij})_{1 \leq i,j \leq d}:= \frac{A+A^T}{2}$ and $C:=\frac{A-A^T}{2}$. If $\tilde{a}_{ij} \in VMO_{\omega}$ for all $1 \leq i,j \leq d$, ${\rm div} \tilde{A} \in L^p(U, \mathbb{R}^d)$ and ${\rm div}\,C \in L^p(U, \mathbb{R}^d)$, then $u \in H^{2, \beta}(U) \cap H^{1,2}_0(U)_b$ and the following estimate holds:
$$
\|u\|_{H^{2,\beta}(U) } \leq \tilde{K}_{16} \|f\|_{L^q(U)},
$$
where $\tilde{K}_{16}>0$ is a constant which only depends on on $d$, $\lambda$, $M$, $r$, $p$, $s$, $q$, $\|h\|_{L^p(U)}$, $\|\tilde{c}\|_{L^s(U)}$ and $\omega$.

\item[(iv)]
Let $\tilde{A}=(\tilde{a}_{ij})_{1 \leq i,j \leq d} := \frac{A+A^T}{2}$ and $C:=\frac{A-A^T}{2}$.  If $\tilde{a}_{ij} \in VMO_{\omega}$ for all $1 \leq i,j \leq d$, ${\rm div} \tilde{A} \in L^d(U, \mathbb{R}^d)$ and ${\rm div}\,C \in L^d(U, \mathbb{R}^d)$ with $d \geq 3$ and $\beta<d$, then $u \in H^{2, \beta}(U) \cap H^{1,2}_0(U)_b$ and the following estimate holds:
$$
\|u\|_{H^{2,\beta}(U) } \leq \tilde{Q}_2 K_{17} \|f\|_{L^q(U)},
$$
where $\tilde{Q}_2>0$ is a constant independent of $f$, $c$, $\tilde{c}$ and $u$ and $K_{17}>0$ is a constant as in (ii).
\end{itemize}
\end{theo}
\begin{proof}
(i) Let $\tilde{g}:= f-cu +\langle {\rm div} A-\mathbf{H}, \nabla u \rangle \in L^{\beta \wedge \frac{2p}{p+2}}(U)$. \\ \\
{\sf \underline{Case 1})} Assume that $\beta \in (1, \frac{2p}{p+2}]$ (i.e. $\frac{1}{\beta} \geq \frac12 +\frac{1}{p}$). Then, $\tilde{g} \in L^{\beta}(U)$. By \cite[Theorem 8]{DK11}, there exists $\tilde{u} \in H^{2, \beta}(U) \cap H_0^{1, \frac{d\beta}{d-\beta}}(U)$ and a constant $\lambda_0>0$ which only depends on $d$, $\lambda$, $M$, $\beta$ and $\omega$ such that
\begin{equation} \label{nondiveqvmo}
- \text{trace}(A \nabla^2 \tilde{u}) + \lambda_0 \tilde{u} = \tilde{g} + \lambda_0 \, u
\end{equation}
and that
\begin{align*}
\|\tilde{u}\|_{H^{2, \beta}(U)} &\leq \bar{C} \| \tilde{g} + \lambda_0 u \|_{L^{\beta}(U)} \\
&\leq \bar{C} \Big( |U|^{\frac{1}{\beta}-\frac{1}{q}} \|f\|_{L^q(U)} + \big(|U|^{\frac{1}{\beta}-\frac{1}{s}} \|\tilde{c}\|_{L^s(U)} +\lambda_0 |U|^{\frac{1}{\beta}} \big)\|u\|_{L^{\infty}(U)} +2|U|^{\frac{1}{\beta}-\frac{1}{2}-\frac{1}{p}}
\|h\|_{L^p(U)} \|\nabla u \|_{L^2(U)}\Big)  \\
& \leq \bar{C} K _{15}\|f\|_{L^q(U)} ,
\end{align*}
where 
$K_{15}:=|U|^{\frac{1}{\beta}-\frac{1}{q}} + \big(|U|^{\frac{1}{\beta}-\frac{1}{s}} \|\tilde{c}\|_{L^s(U)} +\lambda_0 |U|^{\frac{1}{\beta}}  \big)K_{6} + 2\|h\|_{L^p(U)} |U|^{\frac{1}{\beta}-\frac{1}{2}-\frac{1}{p}+\frac{1}{2_*}-\frac{1}{q}} K_5$,\; $2_*$ is an arbitrarily fixed number in $(1, 2 \wedge q)$ if $d=2$, $K_5, K_{6}>0$ are constants as in Theorem \ref{mainexisthm} and $\bar{C}>0$ is a constant which only depends on $d$, $\lambda$, $M$, $r$, $\beta$ and $\omega$.
 Meanwhile, $u \in H^{1,2}_0(U)_b$
satisfies that
\begin{equation*} 
\int_{U} \langle A \nabla u, \nabla \psi \rangle dx  + \int_{U} \langle u\, \text{div} A, \nabla \psi \rangle dx   + \int_{U} \lambda_0 u \psi dx = \int_{U} \big(\tilde{g}+\lambda_0 u \big) \psi dx, \quad \forall \psi \in C_0^{\infty}(U).
\end{equation*}
Likewise, $\tilde{u} \in H^{\frac{d\beta}{d-\beta}}_0(U)$ fulfills from \eqref{nondiveqvmo} and Proposition \ref{fdledivnondi} that
\begin{equation*} 
\int_{U} \langle A \nabla \tilde{u}, \nabla \psi \rangle dx  + \int_{U} \langle \tilde{u}\, \text{div} A, \nabla \psi \rangle dx   + \int_{U} \lambda_0 \tilde{u} \psi dx = \int_{U} \big(\tilde{g}+\lambda_0 u \big) \psi dx, \quad \forall \psi \in C_0^{\infty}(U).
\end{equation*}
Finally applying the uniqueness result in \cite[Theorem 2.1(ii)]{KK19}, we get again $u=\tilde{u} \in H_0^{1,2 \wedge \frac{d\beta}{d-\beta}}(U)$ in $U$, as desired.  \\ \\
{\sf \underline{Case 2})} Assume that $\beta \in (\frac{2p}{p+2}, d]$ (i.e. $\frac{1}{d} \leq \frac{1}{\beta}<\frac{1}{2} + \frac{1}{p}$).
 Choose a unique $k \in \mathbb{N} \cup \{0\}$ such that
$$
\frac{1}{\beta} \in \left[	\frac{1}{2} +\frac{1}{p} + (k+1)
\left( - \frac{1}{d}+\frac{1}{p}\right), \;\; \frac{1}{2} +\frac{1}{p} + k
\left( - \frac{1}{d}+\frac{1}{p}\right)	\right).
$$
Now we claim that for each $m \in \{0,\ldots, k\}$, $u \in H_0^{1, \big( \frac{1}{2} +\frac{1}{p} + m
\left( - \frac{1}{d}+\frac{1}{p}\right) -\frac{1}{d}	\big)^{-1}}(U)$ and there exists a constant $\tilde{K}_m>0$ which only depends on $d$, $s$, $p$, $q$,  $r$, $\lambda$, $M$, $\|h\|_{L^p(U)}$, $\|\tilde{c}\|_{L^s(U)}$, such that
$$
\|u\|_{H_0^{1, \big( \frac{1}{2} +\frac{1}{p} + m
\left( - \frac{1}{d}+\frac{1}{p}\right)-\frac{1}{d}	\big)^{-1}}(U)} \leq \tilde{K}_{m} \|f\|_{L^q(U)}.
$$
Let us first consider the case of $m=0$ to show the claim. By \cite[Theorem 8]{DK11} and the calculation as in {\sf Case 1}, there exists $\tilde{u}_0 \in H^{2, \big( \frac{1}{2} + \frac{1}{p} \big)^{-1}}(U) \cap H_0^{1, \big( \frac{1}{2} +\frac{1}{p} -\frac{1}{d}	\big)^{-1}}(U)$ 
and a constant $\lambda_0>0$ which only depends on $d$, $\lambda$, $M$, $r$, $p$,  $s$, $q$ and $\omega$, such that
\begin{equation*} 
-\text{trace}(A \nabla^2 \tilde{u}_0) + \lambda_0 \tilde{u}_0 = \tilde{g} + \lambda_0 u
\end{equation*}
and that
$$
\|\tilde{u}_0 \|_{H_0^{1, \big( \frac{1}{2} +\frac{1}{p} -\frac{1}{d}	\big)^{-1}}(U)}\leq N_{0,d,p} \|\tilde{u}_0\|_{H^{2,\left( \frac12 +\frac{1}{p} \right)^{-1}}(U)} \leq \tilde{K} _{0} \|f\|_{L^q(U)},
$$
where $N_{0,d,p}>0$ is a constant which only depends on $d$ and $p$ and $\tilde{K} _{0}>0$ is a constant which only depends on $d$, $\lambda$, $M$, $r$, $p$, $s$, $q$, $\|h\|_{L^p(U)}$, $\|\tilde{c}\|_{L^s(U)}$ and $\omega$. As the argument in {\sf Case 1}, we get $u=\tilde{u}_0 \in H^{1,2}_0(U)$, so that the claim holds for $m=0$. Now suppose that $k \geq 1$ and the claim holds for some $m \in \{0, \ldots, k-1 \}$. Then, $\tilde{g}  =  f-cu +\langle {\rm div} A-\mathbf{H}, \nabla u \rangle  \in L^{\big( \frac{1}{2} +\frac{1}{p} + (m+1)
\left( - \frac{1}{d}+\frac{1}{p}\right) \big)^{-1}}(U)$.
 Using \cite[Theorem 8]{DK11}, there exists $\tilde{u}_{m+1} \in H^{2,\big( \frac{1}{2} +\frac{1}{p} + (m+1)
\left( - \frac{1}{d}+\frac{1}{p}\right) \big)^{-1}}(U)\cap H_0^{1, \big( \frac{1}{2} +\frac{1}{p} + (m+1)
\left( - \frac{1}{d}+\frac{1}{p}\right) - \frac{1}{d} \big)^{-1}}(U)$ 
and a constant $\lambda_{m+1}>0$ which only depends on $d$, $\lambda$, $M$, $r$. $p$, $s$, $q$ and $\omega$ such that
\begin{equation*} 
-\text{trace}(A \nabla^2 \tilde{u}_{m+1}) + \lambda_{m+1} \tilde{u}_{m+1} = \tilde{g} + \lambda_{m+1} u
\end{equation*}
and
$$
\|\tilde{u}_{m+1} \|_{H_0^{1, \big( \frac{1}{2} +\frac{1}{p} + (m+1)
\left( - \frac{1}{d}+\frac{1}{p}\right) - \frac{1}{d} \big)^{-1}}(U)}\leq N_{m+1,d,p} \|\tilde{u}_{m+1}\|_{H^{2,\big( \frac{1}{2} +\frac{1}{p} + (m+1)
\left( - \frac{1}{d}+\frac{1}{p}\right) \big)^{-1}}(U)} \leq \tilde{K} _{m+1} \|f\|_{L^q(U)},
$$
where $N_{m+1,d,p}>0$ is a constant which only depends on $d$ and $p$ and $\tilde{K}_{m+1}>0$ is a constant which only depends on $d$, $\lambda$, $M$, $r$, $p$, $s$, $q$, $\|h\|_{L^p(U)}$, $\|\tilde{c}\|_{L^s(U)}$ and $\omega$. As in the argument of {\sf Case 1}, we get $u=\tilde{u}_{m+1} \in H^{1,2}_0(U)$, and hence claim holds for $m+1$. Therefore, the claim is proven by the induction. Now observe that $\langle {\rm div} A-\mathbf{H}, \nabla u \rangle  \in L^{\big( \frac{1}{2} +\frac{1}{p} + (k+1)
\left( - \frac{1}{d}+\frac{1}{p}\right) \big)^{-1}}(U)$, \,so that $\tilde{g}=  f-cu +\langle {\rm div} A-\mathbf{H}, \nabla u \rangle \in L^{\beta}(U)$. Analogously to the arguments so far, there exists $\tilde{u}_{k+1} \in H^{2,\beta}(U)\cap H_0^{1,2}(U)$ 
and a constant $\lambda_{k+1}>0$  which only depends on $d$, $\lambda$, $M$, $r$, $p$, $s$, $q$ and $\omega$ such that
\begin{equation*} 
-\text{trace}(A \nabla^2 \tilde{u}_{k+1}) + \lambda_{k+1} \tilde{u}_{k+1} = \tilde{g} + \lambda_{k+1} u
\end{equation*}
and
$$
\|\tilde{u}_{k+1}\|_{H^{2, \beta}(U)} \leq \tilde{K} _{k+1} \|f\|_{L^q(U)},
$$
where $\tilde{K}_{k+1}>0$ is a constant  which only depends on $d$, $s$, $p$, $q$, $r$, $\lambda$, $M$, $\|h\|_{L^p(U)}$, $\|\tilde{c}\|_{L^s(U)}$ and $\omega$. As the argument in {\sf Case 1)}, we get $u=\tilde{u}_{k+1} \in H^{2, \beta}(U) \cap H^{1,2}_0(U)$, as desired.\\ \\
{\sf \underline{Case 3})} Assume that $\beta\in (d, \infty)$. Choose a unique $k \in \mathbb{N} \cup \{0\}$ such that
$$
\frac{1}{d} \in \left[	\frac{1}{2} +\frac{1}{p} + (k+1)
\left( - \frac{1}{d}+\frac{1}{p}\right), \;\; \frac{1}{2} +\frac{1}{p} + k
\left( - \frac{1}{d}+\frac{1}{p}\right)	\right).
$$
Then, analogously to the claim in {\sf Case 2}, we obtain that $u \in H_0^{1, \big( \frac{1}{2} +\frac{1}{p} + k
\left( - \frac{1}{d}+\frac{1}{p}\right) -\frac{1}{d} \big)^{-1}}(U)$ and there exists a constant $\bar{K}_k>0$ which only depends on $d$, $s$, $p$, $q$, $r$, $\lambda$, $M$, $\|h\|_{L^p(U)}$, $\|\tilde{c}\|_{L^s(U)}$ and $\omega$ such that
$$
\|u\|_{H_0^{1, \big( \frac{1}{2} +\frac{1}{p} + k
\left( - \frac{1}{d}+\frac{1}{p}\right)-\frac{1}{d}	\big)^{-1}}(U)} \leq \bar{K}_{k} \|f\|_{L^q(U)} ,
$$
Thus, $\tilde{g}=  f-cu +\langle {\rm div} A-\mathbf{H}, \nabla f \rangle \in L^{\beta \wedge \big( \frac{1}{2} +\frac{1}{p} +(k+1)
\left( - \frac{1}{d}+\frac{1}{p}\right)  \big)^{-1}}(U)$. As the argument in {\sf Case 1)}, we get  $u  \in H^{2, \beta \wedge \big( \frac{1}{2} +\frac{1}{p} +(k+1)
\left( - \frac{1}{d}+\frac{1}{p}\right)  \big)^{-1}}(U) \cap H^{1,2}_0(U)$ and
$$
\|u\|_{H^{2, \beta \wedge \big( \frac{1}{2} +\frac{1}{p} +(k+1)
\left( - \frac{1}{d}+\frac{1}{p}\right)  \big)^{-1}}(U)} \leq \bar{K}_{k+1}\|f\|_{L^q(U)},
$$
where $\bar{K}_{k+1}>0$ is a constant which only depends on $d$, $\lambda$, $M$, $r$ $p$, $s$, $q$, $\|h\|_{L^p(U)}$, $\|\tilde{c}\|_{L^s(U)}$ and $\omega$.
Now choose $t \in (d, \beta)$. Then, $\frac{1}{p} \leq \frac{1}{\beta}<\frac{1}{t}<\frac{1}{d}$.  Since  
$$
\frac{1}{\beta}<\frac{1}{t}<\frac{1}{t}-\frac{1}{p}+\frac{1}{d}, \qquad \;
\frac{1}{2} +\frac{1}{p} +(k+1)\left( - \frac{1}{d}+\frac{1}{p}\right) \leq \frac{1}{d}<\frac{1}{t}-\frac{1}{p}+\frac{1}{d},
$$ 
it holds that $u \in H^{2, \left(\frac{1}{t}-\frac{1}{p}+\frac{1}{d} \right)^{-1}}(U)$, so that $u \in H^{1,\left(\frac{1}{t}-\frac{1}{p} \right)^{-1}}_0(U)$ and
$$
\|u\|_{H^{1,(\frac{1}{t}-\frac{1}{p})^{-1}}(U)} \leq \bar{N}(d,p, |U|) \bar{K}_{k+1} \|f\|_{L^q(U)},
$$
where $\bar{N}(d,p, |U|)>0$ is a constant which only depends on $d$, $p$ and $|U|$.
Thus, we have $\tilde{g}=  f-cu +\langle {\rm div} A-\mathbf{H}, \nabla u \rangle \in L^{t}(U)$, and hence from the argument as in {\sf Case 1)} we discover that 
$u \in H^{2,t}(U) \cap H^{1, \infty}(U) \cap H^{1,2}_0(U)$ and 
$$
\|u\|_{H^{1,\infty}(U) } \leq \bar{K}_{k+2} \|f\|_{L^q(U)},
$$
where $\bar{K}_{k+2}>0$ is a constant which only depends on  $d$, $\lambda$, $M$, $r$, $p$, $q$, $\|h\|_{L^p(U)}$, $\|\tilde{c}\|_{L^s(U)}$ and $\omega$. Finally, note that $\tilde{g}=  f-cu +\langle {\rm div} A-\mathbf{H}, \nabla u \rangle \in L^{\beta}(U)$, and hence from the argument as in {\sf Case 1} we conclude that $u \in H^{2, \beta}(U) \cap H^{1,2}_0(U)$ and
$$
\|u\|_{H^{2,\beta}(U) } \leq \bar{K}_{k+3} \|f\|_{L^q(U)},
$$
where $\bar{K}_{k+3}>0$ is a constant which only depends on $d$, $\lambda$, $M$, $r$, $p$, $q$, $\|h\|_{L^p(U)}$, $\|\tilde{c}\|_{L^s(U)}$ and $\omega$, as desired. \\ \\
(ii)
Define $g:=f-cu \in L^\beta(U)$. Note that $1<\beta<d$ and $d \geq 3$ and $\text{div}A \in L^d(U, \mathbb{R}^d)$, $\bold{H} \in L^d(U, \mathbb{R}^d)$. Thus,
\cite[Theorem 2.2]{KK19} yields the existence of $w \in H^{2, \beta}(U) \cap H^{1,\frac{d\beta}{d-\beta}}_0(U)$ satisfying
\begin{equation*} 
\int_U \langle A \nabla w, \nabla \psi \rangle + \langle \bold{H}, \nabla w \rangle \psi  \,dx = \int_U g \psi \,dx \quad \text{ for all } \psi \in C_0^{\infty}(U).
\end{equation*}
Additionally using \cite[Theorem 2.2]{KK19} and Theorem \ref{mainexisthm}, we get
\begin{align*}
\|w\|_{H^{2,p}(U)} &\leq Q_2 \|g\|_{L^{\beta}(U)} \leq Q_2 \Big( |U|^{\frac{1}{\beta}-\frac{1}{q}} \|f\|_{L^q(U)}  + |U|^{\frac{1}{\beta}-\frac{1}{s}} \| \tilde{c}\|_{L^s(U)} \|u\|_{L^{\infty}(U)}  \Big) \\
& \leq Q_2 K_{17}\|f\|_{L^q(U)},
\end{align*}
where  $Q_2>0$ is a constant independent of  $f$, $c$, $\tilde{c}$ and $u$ as in \cite[Theorem 2.2]{KK19}. Now it is enough to show that $w=u$ in $U$.
Meanwhile, $u \in H_0^{1,2 \wedge \frac{d\beta}{d-\beta}}(U)$ satisfies
\begin{equation*} 
\int_U \langle A \nabla u, \nabla \psi \rangle + \langle \bold{H}, \nabla u \rangle \psi  dx = \int_U g \psi dx \quad \text{ for all } \psi \in C_0^{\infty}(U).
\end{equation*}
Using the uniqueness results in \cite[Theorem 2.1(i)]{KK19}, we finally get $w=u \in H_0^{1,2 \wedge \frac{d\beta}{d-\beta}}(U)$, as desired. \\ \\
(iii) Observe that by Corollary \ref{corolarconv} and Theorem \ref{mainexisthm}, $u$ is also a unique weak solution to \eqref{maineq2} where $A$ and $\bold{H}$ are replaced by 
$\tilde{A}$ and $-{\rm div}\,C + \bold{H} \in L^p(U, \mathbb{R}^d)$. Thus, the assertion follows from Theorem \ref{vmonondivh1p}(i) where $A$ and $\bold{H}$ are replaced by $\tilde{A}$ and $-{\rm div} \,C + \bold{H}$, respectively. \\ \\
(iv) 
By Corollary \ref{corolarconv}, $u \in H_0^{1,2}(U)$ satisfies
\begin{equation} \label{convertingdiv}
\int_U \langle \tilde{A} \nabla u, \nabla \psi \rangle + \langle {\rm div}\,\tilde{A}+\bold{H}, \nabla u \rangle \psi  dx = \int_U g \psi dx \quad \text{ for all } \psi \in C_0^{\infty}(U),
\end{equation}
where $g:=f-cu \in L^\beta(U)$.
Meanwhile, since ${\rm div} \tilde{A}+\bold{H} \in L^d(U, \mathbb{R}^d)$, \,\cite[Theorem 2.2]{KK19} implies the existence of $w \in H^{2, \beta}(U) \cap H^{1,\frac{d\beta}{d-\beta}}_0(U)$ satisfying that
\begin{equation} \label{weakdefzvmodiv2}
\int_U \langle \tilde{A} \nabla w, \nabla \psi \rangle + \langle {\rm div} \tilde{A}+\bold{H}, \nabla w \rangle \psi  \,dx = \int_U g \psi \,dx \quad \text{ for all } \psi \in C_0^{\infty}(U).
\end{equation}
Additionally using \cite[Theorem 2.2]{KK19} and Theorem \ref{mainexisthm}, we get
$$
\|u\|_{H^{2,\beta}(U) } \leq \tilde{Q}_2 K_{17} \|f\|_{L^q(U)},
$$
where  $\tilde{Q}_2>0$ is a constant independent of  $f$, $c$, $\tilde{c}$ and $u$ as in \cite[Theorem 2.2]{KK19}.
Applying the uniqueness results in \cite[Theorem 2.1(i)]{KK19} to \eqref{convertingdiv} and \eqref{weakdefzvmodiv2}, we finially get $w=u \in H_0^{1,2 \wedge \frac{d\beta}{d-\beta}}(U)$, as desired. 
\end{proof}

\begin{rem} \label{slightremar}
In the case of $d \geq 3$, \cite[Theorem 8]{DK11} crucially used in the proof of Theorem \ref{vmonondivh1p}(i) can be replaced with \cite[Theorem 4.4]{CFL93}, and in this case, the lower bound $\lambda_m$ can be replaced with $0$. However, since we want to include the case of $d=2$, we have used \cite[Theorem 8]{DK11} in the proof of Theorem \ref{vmonondivh1p}(i).
\end{rem}

\noindent
Now we are ready to prove our second main result, Theorem \ref{seconmainres},\\
\centerline{}
\noindent
{\bf Proof of Theorem \ref{seconmainres})} Define $\bold{H}:= \widehat{\bold{H}} + {\rm div} A \in L^p(U, \mathbb{R}^d)$. Let $u\in H^{1,2}_0(U)_b$ be a (unique) weak solution to \eqref{maineq2} as in Theorem \ref{mainexisthm}. In particular, $u$ satisfies \eqref{normalcontrac}  by Theorem \ref{mainexisthm}. Moreover, by Theorem \ref{vmonondivh1p}(i), we deduce that $u \in H^{2,\beta}(U) \cap H^{1,2}_0(U)_b$. Applying integration by parts to \eqref{weakdefzerve} we get
\begin{equation} \label{strongintfrom}
\int_{U} \Big(-\text{\rm trace}(A \nabla^2 u) + \langle -{\rm div} A + \bold{H}, \nabla u \rangle + c u \Big) \psi dx =  \int_{U} f \psi dx \quad \text{ for all } \psi \in C_0^{\infty}(U),\;
\end{equation}
and hence $u$ is a strong solution to \eqref{maineqnondiv}.  
Now assume that $v$ is a strong solution to \eqref{maineqnondiv}, i.e. \eqref{strongformiden} holds where $u$ is replaced by $v$. Then, applying integration by parts to \eqref{strongintfrom} where $u$ is replaced by $v$, we get
$$
\int_{U} \langle A \nabla v, \nabla \psi \rangle dx  + \int_{U} \langle \bold{H}, \nabla v \rangle dx +\int_{U}cv \psi dx =  \int_{U} f \psi dx \quad \text{ for all } \psi \in C_0^{\infty}(U).
$$
Using the uniqueness result in Theorem \ref{mainexisthm}, we get $v=u$, as desired. 
{\vskip 0.1cm \hfill$\Box$}

\begin{rem} \label{posteriana}
Here we mention that the contraction estimate \eqref{normalcontrac} with $\theta=2$ is significant in the error analysis of Physics-Informed Neural Networks (see \cite{YL24}). For instance, under the assumption of Theorem \ref{seconmainres}, additionally assume that $\gamma>0$ is a constant, $c \in L^2(U)$ with $c \geq \gamma$ in $U$ and $f \in L^q(U) \cap L^2(U)$. Let $u \in H^{2,2}(U) \cap H^{1,2}_0(U)_b$ be a unique strong solution to \eqref{maineqnondiv} as in Theorem \ref{seconmainres}.  Let $\Phi_A \in H^{2,2}(U) \cap H^{1,2}_0(U)_b$. Then, $u-\Phi_A \in H^{2,2}(U)\cap H^{1,2}_0(U)_b$ 
and
$$
\mathcal{L}[\Phi_A]:=-\text{\rm trace}(A\nabla^2 \Phi_A) + \langle \widehat{\bold{H}}, \nabla \Phi_A \rangle + c\, \Phi_A \in L^2(U).
$$
Observe that by the linearity
$$
-\text{\rm trace}\Big(A(x) \nabla^2 (u-\Phi_A)(x) \Big) + \Big \langle \widehat{\bold{H}}(x), \nabla (u-\Phi_A)(x) \Big \rangle + c(x) u(x) =\Big(f-\mathcal{L}[\Phi_A] \Big)(x), \quad \text{ for a.e. $x \in U$}.
$$
Thus, applying Theorem \ref{seconmainres}(iii) to the above, we obtain that
\begin{equation} \label{pinnerrorest}
\int_{U} |u-\Phi_A|^2 dx \leq \frac{K_1^2}{\gamma^2} \int_U \big(  f- \mathcal{L}[\Phi] \big)^2 dx,
\end{equation}
where $K_1>0$ is a constant as in Theorem \ref{existrho}, which is independent of $\gamma>0$. By training $\Phi$ to minimize a loss function related to $\int_U \big(  f- \mathcal{L}[\Phi] \big)^2 dx$, \,$\Phi_A$ can be a nice $L^2$-approximation of $u$ which has the $L^2$-error estimates \eqref{pinnerrorest}. By \eqref{divtypeequat} and \eqref{imptcontrac} in Remark \ref{imptremaforme}(ii), the constant $K_1>0$ in \eqref{pinnerrorest} can be replaced by any constant $\tilde{K}_1>0$ satisfying
$$
\frac{\max_{\overline{U}}\tilde{\rho}}{\min_{\overline{U}} \tilde{\rho}} \leq \tilde{K_1},
$$
where $\tilde{\rho} \in H^{1,2}(U) \cap C(U)$ is a function satisfying \eqref{infinitesiminv} and $\tilde{\rho}(x)>0$ for all $x \in \overline{U}$
(see Remark \ref{imptremaforme}).
\end{rem}
\centerline{}
\appendix
\section{Auxiliary results} \label{auxilarlse}
The following are auxiliary results essential for proving the main results of this paper. While similar or identical results may be found in other references, the statements and proofs of these auxiliary results are provided here in full detail for the reader's convenience.
\begin{prop} 
Let $U$ be an open subset of $\mathbb{R}^d$, $\bold{E}=(e_1, \ldots, e_d) \in L^1_{loc}(U, \mathbb{R}^d)$ and $B=(b_{ij})_{1 \leq i,j \leq d}$ be a (possibly non-symmetric) matrix of functions in $L^1_{loc}(U)$ such that ${\rm div} B = \bold{E}$ in $U$. For each $1 \leq j \leq d$,
let $\phi_j \in C_0^{\infty}(U)$ and $V$ be a bounded open subset of $\mathbb{R}^d$ with $\text{supp}(\phi_j) \subset V \subset \overline{V} \subset U$ for all $1 \leq j \leq d$. Let $B_n=(b_{ij}^n)_{1 \leq i,j \leq d}$ be a sequence of matrices of functions satisfying that $b_{ij}^n \in C^{1}(\overline{V})$ for all $n \geq 1$, $1 \leq i,j \leq d$ and
\begin{equation*} 
\lim_{m \rightarrow \infty} b_{ij}^n = b_{ij} \; \text{ in } L^1(V), \quad \text{ for all } 1 \leq i,j \leq d.
\end{equation*}
Then,
\begin{equation} \label{limidivequat}
\lim_{n \rightarrow \infty} -\int_{U} \sum_{j=1}^d \Big(\sum_{i=1}^d \partial_i b^n_{ij} \Big) \phi_j\,dx = -\int_{U} \sum_{j=1}^d e_j \phi_j\,dx.
\end{equation}
\end{prop}
\begin{proof}
Observe that from integration by parts
\begin{equation} \label{nequdivint}
-\int_{U} \sum_{j=1}^d \Big(\sum_{i=1}^d \partial_i b^n_{ij} \Big) \phi_j\,dx = \int_{U} \sum_{i,j=1}^d b^n_{ij} \partial_i \phi_j \,dx.
\end{equation}
Letting  $n \rightarrow \infty$ in \eqref{nequdivint}, we get
$$
\lim_{n \rightarrow \infty} -\int_{U} \sum_{j=1}^d \Big(\sum_{i=1}^d \partial_i b^n_{ij} \Big) \phi_j\,dx =\int_{U} \sum_{i,j=1}^d b_{ij} \partial_i \phi_j\, dx = - \int_{U} \sum_{j=1}^d e_{j} \phi_j dx,
$$
as desired.
\end{proof}

\begin{prop} \label{fdledivnondi}
Let $U$ be an open subset of $\mathbb{R}^d$, $A=(a_{ij})_{1 \leq i,j \leq d}$ be a (possibly non-symmetric) matrix of functions in $L^{\infty}_{loc}(U)$ with ${\rm div} A=(e_1, \ldots, e_d) \in L_{loc}^2(U, \mathbb{R}^d)$. Let $u \in L_{loc}^{1}(U)$ with $\nabla u \in L^2_{loc}(U, \mathbb{R}^d)$
and $\partial_{i} \partial_j u \in L^1_{loc}(U)$ for all $1 \leq i,j \leq d$. Then,
$$
\int_{U} \langle A \nabla u, \nabla \phi \rangle dx = -\int_{U} \Big({\rm trace} (A \nabla^2 u) + \langle {\rm div}A, \nabla u \rangle \Big) \phi \, dx, \quad \text{ for all $\phi\in C_0^{\infty}(U)$}.
$$
\end{prop}
\begin{proof}
Let $\phi \in C_0^{\infty}(U)$ be fixed. Take a bounded open subset $V$ of $\mathbb{R}^d$ with $\text{supp}(\phi) \subset V \subset \overline{V} \subset U$. Then, $u \in H^{2,1}(V)$ and $\nabla u \in L^2(V)$, $\text{div} A \in L^2(V, \mathbb{R}^d)$. Using a mollification, take a sequence of functions $(u_m)_{m \geq 1}$ in $C^{\infty}(\overline{V})$ such that
$$
\lim_{m \rightarrow \infty} u_m = u \quad \text{ in } H^{2,1}(V) \quad \text{ and } \quad \lim_{m \rightarrow \infty} \nabla u_m = \nabla u \quad \text{ in } L^{2}(V, \mathbb{R}^d).
$$
Let $A_n=(a_{ij}^n)_{1 \leq i,j \leq d}$ be a sequence of matrices of functions satisfying that $a_{ij}^n \in C_0^{\infty}(\overline{V})$ for all $n \geq 1$, $1 \leq i,j \leq d$ and
$$
\lim_{n \rightarrow \infty} a_{ij}^n = a_{ij} \; \text{ in } L^1(V), \quad \text{ for all } 1 \leq i,j \leq d.
$$
Thus,
\begin{align*}
&\int_{U} \langle A \nabla u, \nabla \phi \rangle dx = \int_{U} \sum_{i,j=1}^d a_{ij} \partial_j u \,\partial_i \phi \,dx  = \lim_{m \rightarrow \infty} \lim_{n \rightarrow \infty} \int_{U} \sum_{i,j=1}^d a^n_{ij} \partial_j u_m \,\partial_i \phi \,dx \\
&= \lim_{m \rightarrow \infty} \lim_{n \rightarrow \infty} -\int_{U} \sum_{i,j=1}^d \text{trace}(A_n \nabla^2 u_m)\,\phi+  \sum_{j=1}^d \Big(\sum_{i=1}^d \partial_i a_{ij}^n\Big) \Big( \phi \partial_j u_m \Big) dx \\
&\underset{\text{by \eqref{limidivequat}}}{=}\lim_{m \rightarrow \infty} -\int_{U} \sum_{i,j=1}^d \text{trace}(A \nabla^2 u_m) \, \phi+  \sum_{j=1}^d e_j \Big( \phi \partial_j u_m \Big) dx \\
&=  -\int_{U} \Big({\rm trace} (A \nabla^2 u) + \langle {\rm div}A, \nabla u \rangle \Big) \phi \, dx.
\end{align*}
\end{proof}

\begin{cor} \label{corolarconv}
Let $U$ be an open subset of $\mathbb{R}^d$, $C=(c_{ij})_{1 \leq i,j \leq d}$ be an anti-symmetric matrix of functions (i.e. $C=-C^T$) in $L^{\infty}_{loc}(U)$ with ${\rm div}\,C \in L_{loc}^2(U, \mathbb{R}^d)$ and $u \in L^1_{loc}(\mathbb{R}^d)$ with $\nabla u \in L^2_{loc}(U, \mathbb{R}^d)$. Then,
$$
\int_{U} \langle C \nabla u, \nabla \phi \rangle dx = - \int_{U} \langle {\rm div} C, \nabla u \rangle \phi\, dx, \quad \text{ for all $\phi \in C_0^{\infty}(U)$}.
$$
\end{cor}
\begin{proof}
Let $\phi \in C_0^{\infty}(U)$ be fixed. Take a bounded open subset $V$ of $U$ with $\text{supp}(\phi) \subset V \subset \overline{V} \subset U$.
Using a mollification, choose a sequence of functions $(u_m)_{m \geq 1}$ in $C^{\infty}(\overline{V})$ such that
$$
\lim_{m \rightarrow \infty} \nabla u_m = \nabla u \quad \text{ in } L^{2}(V, \mathbb{R}^d).
$$
Then using Proposition \ref{fdledivnondi}, 
\begin{align*}
\int_{U} \langle C \nabla u, \nabla \phi \rangle dx& = \lim_{m \rightarrow \infty} \int_{U} \langle C \nabla u_m, \nabla \phi \rangle dx\\
&=\lim_{m \rightarrow \infty}-\int_{U} \Big({\rm trace} (C \nabla^2 u_m) + \langle {\rm div}C, \nabla u_m \rangle \Big) \phi \, dx  \\
&=\lim_{m \rightarrow \infty} -\int_{U} \langle {\rm div}C, \nabla u_m \rangle \phi \, dx= - \int_{U} \langle {\rm div} C, \nabla u \rangle \phi\, dx,
\end{align*}
as desired.
\end{proof}
\centerline{}
\noindent
The following lemma is a direct consequence of Hardy-Littlewood-Sobolev inequality, but we present the statement and its proof for the reader's accessibility.
\begin{lem} \label{rieszsobol}
Let $f \in L^{\beta}(\mathbb{R}^d)$ with $\beta \in (1, d)$ and $d \geq 2$. Then, $\mathcal{T}f$ defined by
$$
\mathcal{T}f(x):= \int_{\mathbb{R}^d} \frac{f(y)}{\|x-y\|^{d-1}} dy
$$
converges absolutely for a.e. $x \in \mathbb{R}^d$. Moreover, $\mathcal{T}f \in L^{\frac{d\beta}{d-\beta}}(\mathbb{R}^d)$ and
\begin{equation*}
\|\mathcal{T}f \|_{L^{\frac{d\beta}{d-\beta}}(\mathbb{R}^d)} \leq c(d, \beta) \|f\|_{L^{\beta}(\mathbb{R}^d)},
\end{equation*}
where $c(d, \beta)>0$ is a constant which only depends on $d$ and $\beta$.
\end{lem}
\begin{proof}
Let $I_{\alpha}$ be a Riesz potential defined as in \cite[(4), Chapter V]{S70}. Then, by \cite[Chapter V, Theorem 1(i)]{S70}, $\mathcal{T}f$ defined by
$$
\mathcal{T}f(x)=  \frac{2\pi^{d/2} \Gamma(1/2)}{\Gamma\left(\frac{d}{2}-\frac{1}{2}  \right)} I_1 f(x).
$$ 
converges absolutely for a.e. $x \in \mathbb{R}^d$, where $\Gamma$ is the gamma function. Moreover, \cite[Chapter V, Theorem 1(ii)]{S70} implies that
$\mathcal{T}f \in L^{\frac{d\beta}{d-\beta}}(\mathbb{R}^d)$ and there exists a constant  $\tilde{c}(d, \beta)>0$ which only depends on $d$ and $\beta$  such that
$$
\|\mathcal{T}f\|_{L^{\frac{d\beta}{d-\beta}}(\mathbb{R}^d)} =\frac{2\pi^{d/2} \Gamma(1/2)}{\Gamma\left(\frac{d}{2}-\frac{1}{2}  \right)} \|I_1 f\|_{L^{\frac{d\beta}{d-\beta}}(\mathbb{R}^d)}  \leq \frac{2\pi^{d/2} \Gamma(1/2)}{\Gamma\left(\frac{d}{2}-\frac{1}{2}  \right)} \tilde{c} (d, \beta) \|f\|_{L^{\beta}(\mathbb{R}^d)},
$$
and hence the assertion follows.
\end{proof}

\begin{theo} \label{essenfuldiveq}
Let $\hat{f} \in L^{\beta}(\mathbb{R}^d)$ with $\beta \in (1, d)$. Then, there exists $\hat{\bold{F}} \in L^{\frac{d\beta}{d-\beta}}(\mathbb{R}^d, \mathbb{R}^d)$ such that
\begin{equation} \label{intidendiveq2}
\int_{\mathbb{R}^d} \langle  \hat{\bold{F}}, \nabla \psi \rangle dx = \int_{\mathbb{R}^d} \hat{f} \psi dx, \quad \text{ for all $\psi \in C_0^{\infty}(\mathbb{R}^d)$}
\end{equation}
and that
\begin{equation} \label{diveqestime}
\| \hat{\bold{F}} \|_{L^{\frac{d\beta}{d-\beta}}(\mathbb{R}^d)} \leq \hat{c}(d, \beta)  \|\hat{f}\|_{L^{\beta}(\mathbb{R}^d)},
\end{equation}
where $\hat{c}(d, \beta)>0$ is a constant which only depends on $d$ and $\beta$.
\end{theo}
\noindent
\begin{proof}
{\sf \underline{Step 1})}: Let $f \in C_0^{\infty}(\mathbb{R}^d)$ and define $\mathcal{N}f$ given by
\[
\mathcal{N}f(x) =
\begin{cases}
\displaystyle \frac{1}{d(d-2) \alpha(d)} \int_{\mathbb{R}^d}\frac{f(y)}{\|x- y\|^{d-2}} dy, & \text{if } d \geq 3, \\[14pt]
\displaystyle -\frac{1}{2\pi} \int_{\mathbb{R}^2} \log(\|x-y\|) f(y) dy, & \text{if } d=2,
\end{cases}
\]
where $\alpha(d):=\frac{\pi^{\frac{d}{2}}}{\Gamma(\frac{d}{2}+1)}$. Then, it follows from \cite[Section 2.1, Theorem 1]{Ev10} that $\mathcal{N}f \in C^2(\mathbb{R}^d)$ and 
\begin{equation} \label{solpoisson}
-\Delta (\mathcal{N}f) =f \; \text{ on } \; \mathbb{R}^d.
\end{equation}
Now as in Lemma \ref{rieszsobol}, we can define for each $g \in L^{\beta}(\mathbb{R}^d)$
$$
\mathcal{S}g(x):=\frac{-1}{d\alpha(d)} \int_{\mathbb{R}^d}\frac{g(y)}{|x- y|^{d-1}} \frac{x-y}{\|x-y\|} dy, \quad \text{ a.e. }  x \in \mathbb{R}^d.
$$
Then, by Lemma \ref{rieszsobol}, $\mathcal{S}g \in L^{\frac{d\beta}{d-\beta}}(\mathbb{R}^d, \mathbb{R}^d)$ and
\begin{equation} \label{sobolrieszha}
\| \mathcal{S}g \|_{L^{\frac{d\beta}{d-\beta}}(\mathbb{R}^d)} \leq \frac{c(d, \beta)}{d\alpha(d)} \|g\|_{L^{\beta}(\mathbb{R}^d)},
\end{equation}
where $c(d, \beta)>0$ is a constant as in Lemma \ref{rieszsobol}.
Meanwhile, $\nabla (\mathcal{N}f) \in C^1(\mathbb{R}^d, \mathbb{R}^d)$ and $\nabla (\mathcal{N}f) = \mathcal{S}f$ a.e. on $\mathbb{R}^d$, and hence \eqref{solpoisson} yields
\begin{equation} \label{weaksolpoisso}
\int_{\mathbb{R}^d}  \langle \mathcal{S}f, \nabla \psi \rangle dx =  \int_{\mathbb{R}^d} f \psi dx, \quad \text{ for all $\psi \in C_0^{\infty}(\mathbb{R}^d)$}.
\end{equation}
\\
{\sf \underline{Step 2})}:
Now let $\hat{f} \in L^{\beta}(\mathbb{R}^d)$ with $\beta \in (1, d)$ and define $\hat{\bold{F}}:=\mathcal{S}\hat{f} \in L^{\frac{d\beta}{d-\beta}}(\mathbb{R}^d, \mathbb{R}^d)$. Then, \eqref{sobolrieszha} yield \eqref{diveqestime}.
Let $(\hat{f}_n)_{n \geq 1}$ be a sequence of functions  in $C_0^{\infty}(\mathbb{R}^d)$ such that
$$
\lim_{n \rightarrow \infty} \hat{f}_n = \hat{f} \quad \text{ in } L^{\beta}(\mathbb{R}^d). 
$$
Then, \eqref{sobolrieszha} implies that $\mathcal{S} \hat{f}_n \in L^{\frac{d\beta}{d-\beta}}(\mathbb{R}^d, \mathbb{R}^d)$ and
$$
\lim_{n \rightarrow \infty} \mathcal{S} \hat{f}_n = \mathcal{S} \hat{f}=\hat{\bold{F}} \quad \text{ in } L^{\frac{d\beta}{d-\beta}}(\mathbb{R}^d, \mathbb{R}^d).
$$
Letting $n \rightarrow \infty$ in \eqref{weaksolpoisso} where $f$ is replaced by $\hat{f}_n$, we finally get \eqref{intidendiveq2}.
\end{proof}

\begin{cor} \label{essenbdddiveq}
Let $U$ be an open subset of $\mathbb{R}^d$ and let $\hat{f} \in L^{\beta}(U)$ with $\beta \in (1,d)$. Then there exists a vector field $\hat{\bold{F}} \in  L^{\frac{d\beta}{d-\beta}}(U, \mathbb{R}^d)$ such that
\begin{equation} \label{intidendiveqbd}
\int_{U} \langle \hat{\bold{F}}, \nabla \psi \rangle = \int_{U}  \hat{f} \psi dx, \quad \text{ for all $\psi \in C_0^{\infty}(U)$}
\end{equation}
and that 
\begin{equation*} \label{diveqestibdd}
\| \hat{\bold{F}} \|_{L^{\frac{d\beta}{d-\beta}}(U)} \leq \hat{c}(d, \beta)  \|\hat{f}\|_{L^{\beta}(U)},
\end{equation*}
where $\hat{c}(d, \beta)>0$ is a constant as in Theorem \ref{essenfuldiveq} which only depends on $d$ and $\beta$.
\end{cor}
\begin{proof}
Extend $\hat{f}\in L^{\beta}(U)$ on $\mathbb{R}^d$ by the zero-extension and say again $\hat{f} \in L^{\beta}(\mathbb{R}^d)$. Using Theorem \ref{essenfuldiveq}, there exists $\hat{\bold{F}} \in L^{\frac{d\beta}{d-\beta}}(\mathbb{R}^d, \mathbb{R}^d)$ such that \eqref{intidendiveq2} holds, and hence \eqref{intidendiveqbd} follows. Moreover, \eqref{diveqestime} yields
$$
\| \hat{\bold{F}} \|_{L^{\frac{d\beta}{d-\beta}}(U)} \leq \| \hat{\bold{F}} \|_{L^{\frac{d\beta}{d-\beta}}(\mathbb{R}^d)}  \leq  \hat{c}(d, \beta)  \|\hat{f}\|_{L^{\beta}(\mathbb{R}^d)} = \hat{c}(d, \beta)  \|\hat{f}\|_{L^{\beta}(U)},
$$
as desired.
\end{proof}
\centerline{}
The following is a refinement of \cite[Theorem 2.8]{K07} with the help of Theorem \ref{essenfuldiveq}.  
\begin{theo}[Krylov] \label{krylovmainthm}
Let $A=(a_{ij})_{1 \leq i,j \leq d}$ be a (possibly non-symmetric) matrix of functions satisfying \eqref{elliptici} and assume that $a_{ij} \in VMO_{\omega}$ for all $1 \leq i,j \leq d$. Let $f \in L^{\frac{\gamma d}{\gamma+d}}(U)$ and $\bold{F} \in L^\gamma(U, \mathbb{R}^d)$ with $\gamma \in (d, \infty)$. Then, there exists a constant $\tilde{\lambda}>0$ which only depends on $d$, $\lambda$, $M$, $\gamma$ and $\omega$
and there exists a unique $u \in H^{1,\gamma}(\mathbb{R}^d)$ such that
$$
\int_{\mathbb{R}^d} \langle A \nabla u, \nabla \varphi \rangle dx  + \int_{\mathbb{R}^d} \tilde{\lambda} u \psi dx = \int_{\mathbb{R}^d} f \psi dx  + \int_{\mathbb{R}^d} \langle \mathbf{F}, \nabla \psi \rangle dx, \quad \text{ for all } \psi \in C_0^{\infty}(\mathbb{R}^d)
$$
and that
$$
\|u\|_{H^{1,\gamma}(\mathbb{R}^d)} \leq \tilde{N} \left( \|f\|_{L^{\frac{\gamma d}{\gamma+d}}(\mathbb{R}^d)} + \| \bold{F} \|_{L^\gamma(\mathbb{R}^d)} \right),
$$
where $\tilde{N}>0$ is a constant which only depends on $d$, $\lambda$, $M$, $\gamma$ and $\omega$.
\end{theo}
\begin{proof}
By Theorem \ref{essenfuldiveq}, there exists $\hat{\bold{F}} \in L^{\gamma}(\mathbb{R}^d, \mathbb{R}^d)$ and a constant  $\bar{c}(d, \gamma)>0$ is a constant which only depends on $d$ and $\gamma$ such that 
\begin{equation*} \label{diveqestime2}
\| \hat{\bold{F}} \|_{L^{\gamma}(\mathbb{R}^d)} \leq \bar{c}(d, \gamma)  \|f\|_{L^{\frac{\gamma d}{\gamma+d}}(\mathbb{R}^d)},
\end{equation*}
and that
\begin{equation*} \label{intidendivkr}
\int_{\mathbb{R}^d} \langle  \hat{\bold{F}}, \nabla \psi \rangle dx = \int_{\mathbb{R}^d} f \psi dx, \quad \text{ for all $\psi \in C_0^{\infty}(\mathbb{R}^d)$}.
\end{equation*}
Using \cite[Theorem 2.8]{K07},
there exists a constant $\tilde{\lambda}>0$ which only depends on $d$, $\lambda$, $M$, $\gamma$ and $\omega$ and there exists a unique $u \in H^{1,\gamma}(\mathbb{R}^d)$ 
such that
$$
\int_{\mathbb{R}^d} \langle A \nabla u, \nabla \varphi \rangle dx  + \int_{\mathbb{R}^d} \tilde{\lambda} u \psi dx = \int_{\mathbb{R}^d} \langle \hat{\mathbf{F}} + \mathbf{F}, \nabla \psi \rangle dx=\int_{\mathbb{R}^d} f \psi dx  + \int_{\mathbb{R}^d} \langle \mathbf{F}, \nabla \psi \rangle dx, \quad \forall \psi \in C_0^{\infty}(\mathbb{R}^d)
$$
and that
\begin{align*}
\|u\|_{H^{1,\gamma}(\mathbb{R}^d)} &\leq N \left( \|\hat{\bold{F}}\|_{L^{\gamma}(\mathbb{R}^d)} + \| \bold{F} \|_{L^{\gamma}(U)}	\right) \\
& \leq \tilde{N} \left( \|f\|_{L^{\frac{\gamma d}{\gamma+d}}(\mathbb{R}^d)} + \| \bold{F} \|_{L^{\gamma}(U)}	\right),
\end{align*}
where $\tilde{N}:= N \Big( \bar{c}(d, \gamma) \vee N \Big)$ and  $N>0$ is a constant which only depends on $d$, $\lambda$, $M$, $\gamma$ and $\omega$. Thus, the proof is complete.
\end{proof}

\begin{prop} \label{propapprobdd}
Let $v \in H^{1,2}_0(U)_b$. Then, there exist a constant $M>0$ which only depends on $\|v\|_{L^{\infty}(U)}$ and a sequence of functions $(v_n)_{n \geq 1}$ in $C_0^{\infty}(U)$ such that $\sup_{n \geq 1} \|v_n \| \leq M$,
 $\lim_{n \rightarrow \infty} v_n = v$ a.e. on $U$ and $\lim_{n \rightarrow \infty} v_n = v$ in $H^{1,2}_0(U)$.
\end{prop}
\begin{proof}
Let $\eta \in C_0^{\infty}(\mathbb{R})$ be a function satisfying $\eta(t)=t$ if $|t| \leq \|v\|_{L^{\infty}(U)}+1$ and let $M:=\| \eta\|_{L^{\infty}(\mathbb{R})}$.
Since $v \in H^{1,2}_0(U)$ there exists a sequence of functions $(\tilde{v}_n)_{n \geq 1}$ in $C_0^{\infty}(U)$ such that $\lim_{n \rightarrow \infty} \tilde{v}_n = v$ in $H^{1,2}_0(U)$ and $\lim_{n \rightarrow \infty} \tilde{v}_n =v$ a.e. on $U$. Now, define
for each $n \geq 1$ $ v_n:=\eta \circ \tilde{v}_n$ in $U$. Then, $v_n \in C_0^{\infty}(U)$ with $\| v_n \|_{L^{\infty}(U)} \leq M$ for all $n \geq 1$ and $\lim_{n \rightarrow \infty} v_n =v$ a.e. on $U$. Moreover, it follows from the Lebesgue dominated convergence theorem that $\lim_{n \rightarrow \infty} \eta(u_n) = \eta(v) = v$ in $L^2(U)$ and that
$$
\| \nabla v_n - \nabla v \|_{L^2(U)}  \leq \| \eta'(u_n) \nabla u_n - \eta'(u_n) \nabla u \|_{L^2(U)} + \| \eta'(u_n) \nabla u - \eta'(u) \nabla u \|_{L^2(U)} \rightarrow 0 \quad \text{ as $n \rightarrow \infty$},
$$
as desired.
\end{proof}

\begin{prop} \label{sandwich}
Let $B$ be an open ball in $\mathbb{R}^d$, and let $f \in H^{1,2}(B)$ and $g \in H^{1,2}_0(B)$ be such that
$$
0 \leq f \leq g, \quad \text{ a.e. on $B$}.
$$
Then, $f \in H^{1,2}_0(B)$.
\end{prop}
\begin{proof}
First, observe that $g \in H^{1,2}_0(B)$, and hence there exists a sequence $(g_n)_{n \geq 1} \subset C_0^{\infty}(B)$ such that $\lim_{n \rightarrow \infty} g_n = g$ in $H^{1,2}(B)$. Thus, there exists a constant $M > 0$ such that
$$
\sup_{n \geq 1} \|g_n\|_{H^{1,2}(B)} \leq M.
$$
For each $n \geq 1$, define $f_n := f \wedge g_n$. Then, by \cite[Theorem 4.4(iii)]{EG15}, for each $n \geq 1$
we have
$$
f_n = \frac{1}{2} \left( |f + g_n| - |f - g_n| \right) \in H^{1,2}(B),
$$
and hence
\begin{align*} 
\|f_n\|_{H^{1,2}(B)} &\leq \frac12 \Big \| |f + g_n| \Big\|_{H^{1,2}(B)} + \frac12 \Big \| |f - g_n| \Big\|_{H^{1,2}(B)}  \\[5pt]
&= \frac12 \|f + g_n\|_{H^{1,2}(B)} + \frac12 \|f - g_n\|_{H^{1,2}(B)} \\[5pt]
&\leq \|g_n\|_{H^{1,2}(B)} + \|f\|_{H^{1,2}(B)} \leq M + \|f\|_{H^{1,2}(B)}.  
\end{align*}
Thus, 
\begin{equation}\label{approxi}
\sup_{n \geq 1} \|f_n\| \leq M + \|f\|_{H^{1,2}(B)}.
\end{equation}
Since $f_n \in H^{1,2}(B)$ and $f_n$ has compact support in $B$ for each $n \geq 1$, we find that $f_n \in H^{1,2}_0(B)$ for each $n \geq 1$. Moreover, $\lim_{n \rightarrow \infty} f_n = f$ a.e. on $B$. Therefore, by applying the weak compactness of $H^{1,2}_0(B)$ to \eqref{approxi}, we conclude that $f \in H^{1,2}_0(B)$.
\end{proof}
\text{}\\ \\ \\
\noindent
{\bf Acknowledgment.}\;
The author is grateful to the anonymous referees for their constructive comments and valuable suggestions, which improved the quality of this paper.

\text{}\\

Haesung Lee\\
Department of Mathematics and Big Data Science,  \\
Kumoh National Institute of Technology, \\
Gumi, Gyeongsangbuk-do 39177, Republic of Korea, \\
E-mail: fthslt@kumoh.ac.kr, \; fthslt14@gmail.com
\end{document}